\documentclass[12pt,a4paper]{article}
\usepackage{amssymb,amsfonts,amsmath,amsthm}
\usepackage[utf8]{inputenc}
\usepackage[T1]{fontenc}
\usepackage{indentfirst}
\usepackage[dvips]{graphicx}
\usepackage{dsfont}
\usepackage[dvips]{graphics}
\usepackage[dvips]{graphicx}
\usepackage{dsfont}
\usepackage{geometry}
\usepackage{multido}
\usepackage{pstricks}
\usepackage{pstricks-add}
\usepackage{pst-plot}
\usepackage{accents}

\newtheorem{theorem}{Theorem}[section]
\newtheorem{definition}[theorem]{Definition}
\newtheorem{proposition}[theorem]{Proposition}

\newtheorem{lemma}[theorem]{Lemma}
\newtheorem{corollary}[theorem]{Corollary}
\newtheorem{example}[theorem]{Example}

\newtheorem{remark}[theorem]{Remark}

\def \dif {\neq}
\def \meig {\leqslant}

\def \cont {\subseteq}

\def \al {\alpha}
\def \be {\beta}
\def \ga {\gamma}

\def \Z {\mathds{Z}}

\def \> {\rightarrow}

\def\ZZ{{\mathbb Z}}

\def\al{\alpha}
\def\be{\beta}
\def\ga{\gamma}

\def\Si{\Sigma}    \def\si{\sigma}
\def\La{\Lambda}   

    \def\om{\omega}
\def\Ga{\Gamma}

   \def\cQ{{\cal Q}}

\begin{document}

\pagestyle{plain}   

\pagenumbering{arabic}  

\begin{center}
\large{ Simplicity of Partial Crossed Products}\footnote{The third named author
was partially supported by Conselho Nacional de Desenvolvimento Cient\'{i}fico
e Tecnol\'{o}gico (CNPq, Brazil).

{\bf Mathematics Subject Classification:} 16W22; 16D60;  16S35; 37C85; 37BXX; 37B05; 46L05.}
\end{center}

\begin{center}{\bf {\rm Alexandre $Baraviera^1$}, {\rm Wagner $Cortes^2$},
{\rm Marlon $Soares^{3}$}} \end{center}

\begin{center}{\footnotesize $^{1,2,3}$  Instituto de Matem\'{a}tica\\
Universidade Federal do Rio Grande do Sul\\
91509-900, Porto Alegre, RS, Brazil\\
e-mail: {\it atbaraviera@gmail.com}, {\it wocortes@gmail.com},
{\it professormarlonsoares@gmail.com}}
\end{center}

\begin{center} {\footnotesize $^{3}$ Departamento de Matem\'{a}tica\\
Universidade Estadual do Centro-Oeste\\
Guarapuava, PR, Brazil\\
e-mail: {\it marlonsoares@unicentro.br}}
\end{center}

\begin{abstract}
In this article, we consider a twisted partial action $\alpha$ of a group $G$ on a ring
$R$ and it is associated partial crossed product $R*_{\alpha}^wG$. We study necessary
and sufficient conditions for the commutativity and simplicity of $R*_{\alpha}^wG$.

Let $R=C(X)$ the algebra of continuous functions  of a topological space $X$ on the
complex numbers and $C(X)*_{\alpha}G$ the partial skew group ring, where $\alpha$
is a partial action of a topological group  $G$ on $C(X)$.  We study some topological properties to
obtain results on the algebra $C(X)$. Also, we study the simplicity of
$C(X)*_{\alpha}G$ using topological properties and the results about the simplicity
of partial crossed product obtained for $R*_{\alpha}^wG$.
\end{abstract}

\section*{Introduction}

Partial actions of groups have been introduced in the theory of operator algebras
as a general approach to study $C^{*}$-algebras by partial isometries (see, in
parti\-cular, (\cite{E1}–\cite{E3})), and crossed products classically, as
well-pointed out in \cite{DES1}, “are the center of the rich interplay between
dynamical systems and operator algebras” (see, for instance, \cite{M1} and
\cite{Q1}). The general notion of a (continuous) twisted partial action of a
locally compact group on a $C^{*}$-algebra and the corresponding crossed product
were introduced in \cite{E1}. Algebraic counterparts for some of the above
mentioned notions were introduced and studied in \cite{DE}, stimulating further
investigations, for instance,  see \cite{FL},  and references therein.  In
particular, twisted partial actions of groups on abstract rings and corresponding
partial crossed products were recently introduced in \cite{DES1}.

In this article, we work with twisted partial actions  and partial crossed products
in the sense of \cite{DES1}.    We study necessary and sufficient conditions for
the partial crossed product to be simple.  We apply these results to obtain necessary
and sufficient conditions for the partial crossed product over the algebra of the
all continuous functions of a topological space on the complex numbers to be simple.

The paper is organized as follows:

In the Section 1, we give some algebraic and topological notions that will be used
in this article.

In the Section 2, we consider the partial crossed product in the sense of \cite{DES1}.
We completely describe the center and we study the commutativity of the partial
crossed product. We  study, in a lot of cases,  necessary and sufficient conditions
for the simplicity of the partial crossed product and  this generalizes the results
presented in \cite{JO} and \cite{JO1}.

In the Section 3, we consider  the algebra $C(X)$, that is the algebra of the
continuous functions of a topological space $X$ on the complex numbers. We consider
a partial action of a topological group $G$ on $X$ and its extension to $C(X)$. We study
some topological properties of the partial action of $G$ on  $X$ that will imply some
algebraic properties on $C(X)$. Also, we   apply the
results of the Section 2 to study  the simplicity of the partial crossed product
over $C(X)$.

In the Section 4, we present some topological examples where we apply the results
of the Section  3.

\section{Preliminaries}

Let $A$ be an associative non-necessarily unital ring, we remind that the ring of
multipliers $\mathcal{M}(A)$  is the set
\begin{center}
$\mathcal{M}(A) = \{(R,L) \in End(_{A}A)\times End(A_A):(aR)b=a(Lb)
 \text{ for all } a,b \in A\}$
\end{center}
with the following operations:
\begin{itemize}
\item[$(i)$] $(R,L)+(R',L')= (R+R',L+L')$;
\item[$(ii)$] $(R,L)(R',L')= (R'\circ R,L\circ L')$.
\end{itemize}

Here we use the right hand side notation for homomorphisms of left $A$-modules,
while for homomorphisms of right modules the usual notation shall be used. In
particular, we write $a \mapsto aR$ and $a \mapsto La$ for $R:_ {A}A \to _ {A}A$,
$L: A_A \to A_A$ with $a \in A$. For the multiplier $w = (R,L) \in \mathcal{M}(A)$
and $a \in A$ we set $aw=aR$ and $wa = La$. Thus one always has $(aw)b=a(wb)$,
$a,b \in A$. The first (resp. second) components of the elements of
$\mathcal{M}(A)$ are called right (resp. left) multipliers of $A$. It is convenient
to point out that if $A$ is a unital ring, then we have that
$A \simeq \mathcal{M}(A)$, see (\cite{DE}, Proposition 2.3). So, in this case, each
invertible multiplier may be consider as a invertible element of $A$.

The following definition appears in (\cite{DES2}, Definition 2.1).

\begin{definition} \label{def1}
A \emph{twisted partial action} of a group $G$ on a ring $R$ is a triple
\begin{center}
$\alpha = (\{D_g\}_{g \in G}, \{\alpha_g\}_{g \in G},
\{w_{g,h}\}_{(g,h) \in G\times G})$,
\end{center}
where for each $g \in G$, $D_g$ is a two-sided ideal in $R$, $\alpha_g:D_{g^{-1}}
\rightarrow D_g$ is an isomorphism of rings  and for each $(g,h) \in G \times G$,
$w_{g,h}$ is a invertible element from $\mathcal{M}(D_g D_{gh})$, satisfying the
following postulates, for all $g,h,t \in G$:
\begin{itemize}
\item [$(i)$] $D_g^2 = D_g$ and $D_g D _h = D_h D_g$;
\item [$(ii)$] $D_1 = R$ and $\alpha_1$ is the identity map of $R$;
\item [$(iii)$] $\alpha_g(D_{g^{-1}} D_h)= D_g D_{gh}$;
\item [$(iv)$] $\alpha _g \circ \alpha _h (a)= w_{g,h} \alpha_{gh}(a) w_{g,h}^{-1}$,
 $\forall a \in D_{h^{-1}} D_{h^{-1} g ^{-1}}$;
\item [$(v)$] $w_{g,1}=w_{1,g}=1$;
\item [$(vi)$] $\alpha_g(a w _{h,t}) w _{g,ht}= \alpha_g(a)w_{g,h} w _{gh,t}$,
 $\forall a \in D_{g^{-1}} D_{h}D_{ht}$.
\end{itemize}
\end{definition}

Note that if  $w_{g,h} = 1_g1_{gh}$, $\forall g,h\in G$, then we have the partial
action defined by Dokuchaev and Exel in (\cite{DE}, Definition 1.1) and when
$D_g = R$, $\forall g\in G$, we have that $\alpha$ is a twisted global action.

Let $\beta = (T, \{\beta_g\}_{g \in G}, \{u_{g, h}\}_{(g, h) \in G \times G})$ be
a twisted global action of a group $G$ on a (non-necessarily unital) ring $T$ and
$R$ an ideal of $T$ generated by a central idempotent $1_R$. We can restrict
$\beta$ for $R$ as follows: we set $D_g = R\cap \beta_{g}(R) = R\cdot \beta_{g}(R)$
and we have that  $D_{g}$ has identity $1_{g} = 1_R\beta_{g}(1_R)$. We define
$\alpha_{g}=\beta_{g}|_{D_{g^{-1}}}$, $g\in G$, and the items (i), (ii) and (iii) of
Definition \ref{def1} are satisfied. Furthermore, let $w_{g, h} =
u_{g, h}1_R\beta_g(1_R)\beta_{gh}(1_R)$, $g,h\in G$, and we have that (iv), (v) e (vi) are also
satisfied. So, we indeed have obtained a twisted partial action of $G$ on $R$.

The following definition appears in (\cite{DES2}, Definition 2.2).

\begin{definition}\label{def2}
A \emph{twisted global action}
\begin{center}
$(T, \{\beta_g\}_{g\in G}, \{u_{g,h}\}_{(g,h)\in G\times G})$
\end{center}
of a group $G$ on an associative \emph{(}non-necessarily unital\emph{)} ring $T$ is
said to be a \emph{globalization} \emph{(}or an \emph{enveloping action}\emph{)} for the partial
action $\alpha$ of $G$ on a ring $R$ if there exists a monomorphism
$\varphi:R\rightarrow T$ such that, for all $g$ and $h$ in $G$:
\begin{itemize}
\item [$(i)$] $\varphi(R)$  is an  ideal of $T$;
\item [$(ii)$] $T = \sum_{g\in G}\beta_g(\varphi(R))$;
\item [$(iii)$] $\varphi(D_g) = \varphi(R)\cap \beta_g(\varphi(R))$;
\item [$(iv)$] $\varphi\circ \alpha_g(a) = \beta_g\circ \varphi(a)$,
 $\forall x\in D_{g^{-1}}$;
\item [$(v)$] $\varphi(aw_{g,h}) = \varphi(a)u_{g,h}$ and
 $\varphi(w_{g,h}a) = u_{g,h}\varphi(a)$, $\forall a\in D_gD_{gh}$.
\end{itemize}
\end{definition}

In (\cite{DES2}, Theorem 4.1), the authors studied necessarily and sufficient
conditions for a twisted partial action $\alpha$ of a group $G$ on a ring $R$
has enveloping action. Moreover, they studied which rings satisfy such conditions.

Suppose that $(R, \alpha, w)$ has an enveloping action $(T,\beta,u)$. In this case,
we may assume that $R$ is an ideal of $T$ and we can rewrite the conditions of the
Definition \ref{def2} as follows:
\begin{itemize}
\item [$(i')$] $R$  is an  ideal of $T$;
\item [$(ii')$] $T = \sum_{g\in G}\beta_g(R)$;
\item [$(iii')$] $D_g = R\cap \beta_g(R)$, $\forall g\in G$;
\item [$(iv')$] $\alpha_g(a) = \beta_g(a)$, $\forall x\in D_{g^{-1}}$ and
 $\forall g\in G$;
\item [$(v')$] $aw_{g,h} = au_{g,h}$ and $w_{g,h}a = u_{g,h}a$, $\forall a\in
 D_gD_{gh}$ and $\forall g, h \in G$.
\end{itemize}

Given a twisted partial action $\alpha$ of a group $G$ on a ring $R$, we recall
from (\cite{DES1}, Definition 2.2) that the partial crossed product
$R*_{\alpha}^{w}G$ is the direct sum
\begin{center}
$\bigoplus_{g\in G}D_g\delta_g$,
\end{center}
where $\delta_g$'s are symbols, with  usual addition and the
multiplication  rule is
\begin{center}
$(a_g\delta_g)(b_h\delta_h) = \alpha_g
(\alpha^{-1}_g(a_g)b_h)w_{g,h}\delta_{gh}$.
\end{center}

By (\cite{DES1}, Theorem 2.4)  we have that $R*_{\alpha}^{w}G$ is
an associative ring whose identity is $1_R\delta_1$. Moreover,  we
have the injective morphism $\phi: R\rightarrow R*_{\alpha}^
{w}G$, defined by $r\mapsto r\delta_1$ and we can consider
$R*_{\alpha}^{w}G$ an extension of $R$.

We finish this subsection with the following well known definition, see \cite{FL}.

\begin{definition}\label{def3}
Let $\alpha$ be a twisted partial action of a  group $G$ on a ring
$R$. An ideal  $I$ of $R$ is said to be $\alpha$-\emph{invariant} if
$\alpha_g(I\cap D_{g^{-1}}) \subseteq I\cap D_g$, $\forall g\in G$.
\end{definition}

\subsection{Topological notions}

In this subsection, we review some definitions and results on topological spaces
that will be used in the paper.  We begin with the  definition of partial actions 
of topological groups on topological spaces, see  \cite{Ab1}.

\begin{definition} Let $G$ be a topological group and $X$ a topological space.
A \emph{partial action} $\alpha$ of $G$ on $X$ is a family of open subsets
$\{ X_t \}_{t \in G}$ of $X$ and homeomorphisms $\alpha_t \colon X_{t^{-1}}\to X_t$
such that the following properties hold:
\begin{itemize}
\item [$(i)$] $X_e=X$ and $\al_e = id_X$;
\item [$(ii)$] $\al_t(X_{t^{-1}} \cap X_s) = X_t \cap X_{ts}$;
\item [$(iii)$] $\al_t(\al_s(x)) = \al_{ts}(x),
 \forall x\in \al_{s^{-1}}(X_s \cap X_{t^{-1}})$;
\item [$(iv)$] The set $\Ga_{\al} = \{ (t, x) \in G \times X :
 t \in G, x \in X_{t^{-1}}\}$ is open
 in $G \times X$ and the function $\varphi \colon \Ga_{\al} \to X$ defined
 by $\varphi(t, x) = \al_t(x)$ is continuous.
\end{itemize}
\end{definition}

We denote it by the triple $(X, \al, G)$ and  it is called  a \emph{partial
dynamical system}, see \cite{E3}.

Next, we give a non-trivial example of a partial dynamical system.

\begin{example}\label{example4.1}
{\bf (non complete flows)}
\emph{Consider a smooth vector field $\mathfrak{X} \colon X \rightarrow TX$ on a manifold $X$,
and for any $p \in X$ let $\phi_p$ be the flow of $X$ through $p$, i.e. the solution
of the differential equation}
$$
    \frac{d}{dt}\phi_p(t) = \mathfrak{X}\big(\phi_p(t)\big)
$$
\emph{with initial condition $\phi_p(0)=p$, defined on its maximal interval $(a_p, b_p)$.
For any $t \in \mathds{R}$, set $X_{-t} = \{p \in X : t \in (a_p, b_p) \}$,
$\alpha_t \colon X_{-t}\to X_t$ such that $\al_t(p)=\phi_p(t)$, and
$\alpha = (\{X_t\}_{t \in \mathds{R}}, \{\alpha_t\}_{t \in \mathds{R}})$.
Now $(X, \al, \mathbb{R})$ is a partial dynamical system (and, if
the manifold is compact, it is in fact a global action, and so, a usual
dynamical system). See \cite{Ab1}.}
\end{example}

 From (\cite{Ab1}, Theorem 2.5) we have that any partial
action can be seen as the restriction to $X$  of a global action  of $G$ on a
topological space $X^e$ by homeomorphisms denoted by $(X^e, \be, G)$.  For
convenience we briefly recall the construction of (\cite{Ab1},
Theorem 2.5): first, define the action $\ga \colon G \times G
\times X \to G \times X$ as $\ga_s(t, x)=(st, x)$ and introduce the equivalence
relation on $G \times X$ defined as follows:
$$(t, x) \sim (s, y) \Leftrightarrow x\in X_{t^{-1}s}  \text{ and }
\al_{s^{-1}t}(x) = y$$
Then we have the topological space $X^e= G \times X /\sim$ and
denote the equivalence class of $(g, x) \in G \times X$, as usual,
by $[g, x] \in X^e$. The global action $\be$ is just the
restriction of the action $\ga$ to the equivalence classes. The
quotient map $q \colon G \times X \to X^e$ is $q(g, x) = [g, x]$;
it is also possible to introduce the injective morphism $i \colon X \to X^e$
given by $i(x)=q(e, x)$, that is, an injective continuous morphism. 
Moreover, for each $x\in X_{g^{-1}}$, we have that
$i(\alpha_g(x)) = q(e,\alpha_g(x)) = q(g,x) = q(\gamma_g(e,x)) = \beta_{g}(q(e,x)) = \beta_g(i(x))$ 
and $X$ is open in $X^e$. The triple $(X^e, \beta, G)$ is called the 
\emph{enveloping action} of $(X,\alpha, G)$, see \cite{Ab1}.

Let $(X,\alpha, G)$ be a partial dynamical system and consider the algebra of
continuous functions
$$C(X)= \{f:X\rightarrow \mathbb{C}\,\, {\rm continuous}\}$$

\noindent with usual sum of functions and multiplication given by the rule:
$(ff')(x)=f(x)f'(x)$, for any $f, f'\in C(X)$. Following \cite{E3} we can extend
the partial action $\alpha$ of $G$ on $X$ to the algebra $C(X)$ with ideals
$C(X_t)$ and isomorphisms $\alpha_t:C(X_{t^{-1}})\rightarrow C(X_t)$ defined by
$\alpha_t(f)(x)=f(\alpha_{t^{-1}}(x))$ for each $t\in G$ and the following
properties are easy to be verified:
\vskip1mm%
a) $C(X_e)=C(X)$ and $\al_e = id_{C(X)}$
\vskip1mm%
b) $\al_t(C(X_{t^{-1}})) \cap C(X_s)) = C(X_t) \cap C(X_{ts})$
\vskip1mm%
c) $\al_t(\al_s(f)) = \al_{ts}(f), \forall f\in \al_{s^{-1}}(C(X_s)\cap C(X_{t^{-1}}))$
\vskip1mm%
We denote this partial action by $\alpha$ again. Following \cite{DE} the
partial skew group ring $C(X)\ast_{\alpha}G$ is the set of all finite formal
sums $\sum_{g\in G} a_g\delta_g$, where $a_g\in C(X_g)$, with usual sum and
multiplication given by the rule $(a_g\delta_g)(a_h\delta_h)=
\alpha_g(\alpha_{g^{-1}}(a_g)a_h)\delta_{gh}$. Note that, in this case,
$C(X)\ast_{\alpha}G$ is not necessarily a $C^*$-crossed product. Moreover,
$C(X)*_{\alpha}G$ is associative, because $(X,\alpha, G)$ has an enveloping
action $(X^e, \beta, G)$ and in this case $C(X)*_{\alpha}G$ is a subring of
the skew group ring $C(X^e)*G$, see [4].

\section{ Simplicity of Partial Crossed Products}

Let  $R$ be a ring and $X$ a non-empty subset of $R$.  The  \emph{centralizer} of
$X$ in $R$ is the set $C_R(X) = \{r\in R|rx = xr, \forall x\in X\}$. It is easy to
see that $C_R(X)$ is a subring of $R$.  Note that if  $X = R$ the centralizer
$C_R(X)$ is  the center of  $R$ and it is denoted by  $Z(R)$.

From now on, we assume that $R$ is a ring with identity $1_R$ and
\begin{center}
$\alpha = \big(\{ D_g\}_{g \in G}, \{\alpha_g\}_{g \in G}, \{w_{g,h}\}_{(g,h)
 \in G\times G}\big)$
\end{center}
is a twisted partial action of $G$ on $R$ such that all the ideals $D_g$, $g\in G$,  are
generated by central idempotents $1_g$. Note that  this is not sufficient for
a twisted partial action of a group $G$ on a ring $R$ to have an enveloping action,
see (\cite{DES2}, Theorem 4.1).

The next result  was proved in (\cite{PA}, Lemma 2.1) and we put its proof here
for the reader's convenience.

\begin{lemma}\label{lem1.2} Let $\alpha$ be a twisted  partial action of a group
$G$ on a ring $R$. Then
$$C_{R*_{\alpha}^wG}(R) = \{\sum\limits_{g\in G}a_g\delta_g \in
R*_{\alpha}^wG|a_g\alpha_g(r1_{g^{-1}}) = ra_g, \forall r\in R\text{
and } \forall g\in G\}.$$
\end{lemma}
\begin{proof}  Note that, $\forall r\in R \text{ and }\forall g\in G$,
\begin{eqnarray*}
\sum\limits_{g\in G}a_g\delta_g \in C_{R*_{\alpha}^wG}(R) & \Leftrightarrow &
\Big(\sum\limits_{g\in G}a_g\delta_g\Big)(r\delta_e) =
(r\delta_e)\Big(\sum\limits_{g\in G}a_g\delta_g\Big), \forall r\in R\\
& \Leftrightarrow & \sum\limits_{g\in G} a_g\alpha_g(r1_{g^{-1}})w_{g, e}\delta_{ge} =
\sum\limits_{g\in G} r\alpha_e(a_g)w_{e, g}\delta_{eg}, \forall r\in R\\
& \Leftrightarrow & \sum\limits_{g\in G} a_g\alpha_g(r1_{g^{-1}})\delta_g =
\sum\limits_{g\in G} ra_g1_g\delta_g, \forall r\in R\\
& \Leftrightarrow & a_g\alpha_g(r1_{g^{-1}}) = ra_g, \forall r\in R.
\end{eqnarray*}
\end{proof}

Let $R$ be a commutative ring. We denote the annihilator of an element $a\in
R$ by $\text{ann}(a)$. When $R$ is commutative we have the following consequence.

\begin{corollary}\label{cor1.3} Let $\alpha$ be a twisted partial action of a group
$G$ on a commutative ring $R$. Then
\begin{center} $C_{R*_{\alpha}^wG}(R) =
\{\sum\limits_{g\in G}a_g\delta_g \in R*_{\alpha}^wG|\alpha_g(r1_{g^{-1}}) - r1_g \in
\text{\emph{ann}}(a_g), \forall r\in R\text{ and } \forall g\in G\}$.
\end{center}
\end{corollary}
\begin{proof}
Using Lemma  \ref{lem1.2}  and the assumption on $R$, we have that
\begin{eqnarray*}
\sum\limits_{g\in G}a_g\delta_g \in C_{R*_{\alpha}^wG}(R)
& \Leftrightarrow & a_g\alpha_g(r1_{g^{-1}}) = ra_g, \forall r\in R \text{ and }
 \forall g\in G\\
& \Leftrightarrow & a_g\alpha_g(r1_{g^{-1}}) = a_gr1_g, \forall r\in R \text{ and }
 \forall g\in G\\
& \Leftrightarrow & a_g\alpha_g(r1_{g^{-1}}) - a_gr1_g = 0, \forall r\in R
 \text{ and }\forall g\in G\\
& \Leftrightarrow & a_g(\alpha_g(r1_{g^{-1}}) - r1_g) = 0, \forall r\in R
 \text{ and }\forall g\in G\\
& \Leftrightarrow & \alpha_g(r1_{g^{-1}}) - r1_g \in \text{ann}(a_g),
 \forall r\in R\text{ and } \forall g\in G.
\end{eqnarray*}
\end{proof}

We say that  $R$ is \emph{maximal commutative} in $R*_{\alpha}^wG$ if
$R=C_{R*_{\alpha}^wG}(R)$. Note that if $R$ is commutative, then
$R\subseteq C_{R*_{\alpha}^wG}(R)$. Using the Corollary \ref{cor1.3}
we obtain the following result.

\begin{corollary}\label{cor1.4}
Let $\alpha$ be a twisted partial action of a group $G$ on a commutative ring $R$.
Then $R$ is maximal  commutative in $R*_{\alpha}^wG$ if and only if
$\forall g\in G\backslash \{e\}$ and $\forall a_g\in D_g\backslash \{0\}$, there
exists $r\in R$ such that $\alpha_g(r1_{g^{-1}}) - r1_g \notin \text{\emph{ann}}(a_g)$.
\end{corollary}

Using the Corollary \ref{cor1.4}  we have  the following.

\begin{corollary}\label{cor1.5}
Let $\alpha$ be a twisted partial action of $G$ on a commutative ring $R$  and
suppose that for each $g\in G\backslash \{e\}$ there exists $r\in R$ such that
$\alpha_g(r1_{g^{-1}}) - r1_g$ is not a zero divisor in $D_g$. Then $R$ is
maximal commutative in  $R*_{\alpha}^w G$.
\end{corollary}

\begin{proposition}\label{cor1.6}
Let $\alpha$ be a  twisted partial action such that  there exists
$g\in G\backslash \{e\}$ with $D_g\neq \{0\}$.  If $R$ is maximal commutative
in $R*_{\alpha}^wG$ then for each $g\in G$ such that $D_g \neq \{0\}$ we have
that $\alpha_g \neq id_{D_g}$.
\end{proposition}
\begin{proof}
Suppose that there exists $h\in G$ such that  $D_h\neq \{0\}$ with $\alpha_h = id_{D_h}$.
Thus, $D_h = D_{h^{-1}}$ and we have that $1_h = 1_{h^{-1}}$. Let  $a_h\delta_h\neq 0$.
Then for each  $r\in R$,
\begin{center}
$(a_h\delta_h)(r\delta_e) = a_h\alpha_h(r1_{h^{-1}})w_{h, e}\delta_{he} =
a_hr1_{h^{-1}}1_h\delta_{h} = a_hr\delta_{h} = (r\delta_e)(a_h\delta_{h}).$
\end{center}
Hence, $a_h\delta_h\in C_{R*_{\alpha}^wG}(R)$ which contradicts the fact that $R$ is
maximal commutative in $R*_{\alpha}^wG$.
\end{proof}

The following example shows that the assumption in Proposition \ref{cor1.6} is
not superfluous.

\begin{example}
\emph{Let $R$ be a commutative ring and $G$ any group. We define the following partial
action: $D_e=R$, $D_g=(0)$, $\forall g\in G\backslash \{e\}$, $\alpha_e=id_R$ and
$\alpha_g\equiv 0$, $\forall g\in G\backslash \{e\}$. We easily obtain that $R$ is maximal
commutative in $R*_{\alpha}^wG$ and $\alpha_g=id_{D_g}$,
$\forall g\in G\backslash \{e\}$.}
\end{example}

Next, we extend the concept of symmetric cocycle for twisted partial action.

\begin{definition}\label{def1.7}
Let $\alpha = \big(\{D_g\}_{g\in G}, \{\alpha_g\}_{g \in G},
\{w_{g, h}\}_{g, h \in G}\big)$ be a twisted partial action of $G$ on $R$.
We say that  $w$ is \emph{symmetric}  if, $\forall g, h\in G$, the following
conditions are satisfied:
\begin{itemize}
\item [$(i)$] $D_gD_{gh} = D_hD_{hg}$;
\item [$(ii)$] $w_{g, h} = w_{h, g}$.
\end{itemize}
\end{definition}

It is not difficult to show that if  $\alpha = \big(\{D_g\}_{g\in G},
\{\alpha_g\}_{g \in G}, \{w_{g, h}\}_{g, h \in G}\big)$ has enveloping
action $(T,\beta,u)$  such that $u_{g,h}=u_{h,g}$, for all $g,h\in G$,
then $w$ is symmetric.

\begin{corollary}\label{cor1.8}
Let $\alpha$ be a twisted partial action of a group $G$ on a ring $R$. If $R$
is commutative, $G$ is abelian and $w$ is symmetric, then $C_{R*_{\alpha}^wG}(R)$
is commutative.
\end{corollary}
\begin{proof}
Let $x, y \in C_{R*_{\alpha}^wG}(R)$ such that $x = \sum\limits_{g\in G}a_g\delta_g$
and $y = \sum\limits_{h\in G}b_h\delta_h$. By Lemma  \ref{lem1.2}, we have that
$a_g\alpha_g(b_h1_{g^{-1}}) = b_ha_g$ and $b_h\alpha_h(a_g1_{h^{-1}}) = a_gb_h$,
$\forall g, h \in G$. By the fact that  $R$ is commutative, we have that
$a_g\alpha_g(b_h1_{g^{-1}}) = b_h\alpha_h(a_g1_{h^{-1}})$, $\forall g, h \in G$.
Since $G$ is abelian and $w$ is symmetric, we have that
\begin{eqnarray*}
\bigg(\sum\limits_{g\in G}a_g\delta_g\bigg)\bigg(\sum\limits_{h\in G}b_h\delta_h\bigg)
& = & \sum\limits_{g, h\in G}a_g\alpha_g(b_h1_{g^{-1}})w_{g, h}\delta_{gh}\\
& = & \sum\limits_{g, h\in G}b_h\alpha_h(a_g1_{h^{-1}})w_{h, g}\delta_{hg}\\
& = & \bigg(\sum\limits_{h\in G}b_h\delta_h\bigg)\bigg(\sum\limits_{g\in G}a_g\delta_g\bigg).
\end{eqnarray*}
So, $C_{R*_{\alpha}^wG}(R)$ is commutative.
\end{proof}

For each $y=\sum_{g\in G}a_g\delta_g\in R*_{\alpha}^{w}G$,  the \emph{support}
of the element $y$ is the set $\text{supp}(y)=\{g\in G:a_g\neq 0\}$. Moreover,
we denote $|\text{supp}(y)|$ as the cardinality of the support of the element $y$.

\begin{lemma}\label{lem1.9} Let $\alpha$ be a twisted partial action of a group
$G$ on a ring $R$. If $R$ is commutative, then $(I\cap C_{R*_{\alpha}^wG}(R))\neq \{0\}$,
for all nonzero ideal $I$ of $R*_{\alpha}^wG$.
\end{lemma}
\begin{proof}
For each $g\in G$, we define $T_g:R*_{\alpha}^wG \> R*_{\alpha}^wG$ by
$T_g\big(\sum_{h\in G}a_h\delta_h\big) =\big(\sum_{h\in G}a_h\delta_h\big)
\big(1_g\delta_g\big)$. It is easy to verify that $T_g$ is an homomorphism
of left $R*_{\alpha}^wG$-modules such that $T_g(I)\subseteq I$, for each
ideal $I$ of $R*_{\alpha}^wG$ and for each $g\in G$. Note that for each
$0\neq a = \sum_{h\in G}a_h\delta_h$, with $a_e = 0$, there exists
$p\in \text{supp}(a)$ such that
\begin{center}
$c = \sum\limits_{l\in G}c_l\delta_l := T_{p^{-1}}\Big(\sum\limits_{h\in G}a_h\delta_h\Big)
 = \sum\limits_{h\in G}a_h\alpha_h(1_{p^{-1}}1_{h^{-1}})w_{h, {p^{-1}}}\delta_{h{p^{-1}}}$
\end{center}
satisfies $c_e = a_pw_{p, p^{-1}}\neq 0$ and $1\leq |\text{supp}(c)| \leq
|\text{supp}(a)|$.

For each  $r\in R$ we define  $K_r:R*_{\alpha}^wG \> R*_{\alpha}^wG$  by
$K_r\big(\sum_{h\in G}a_h\delta_h\big) = (r\delta_e)\big(\sum_{h\in G}a_h\delta_h\big)
- \big(\sum_{h\in G}a_h\delta_h\big)(r\delta_e)$. It is easy to see that $K_r$
is an homomorphism of additive abelian groups such that $K_r(I)\subseteq I$, for each
ideal $I$ of $R*_{\alpha}^wG$ and for each $r\in R$. Since $R$ is commutative
and $r1_e - \alpha_e(r1_{e^{-1}}) = 0$, we have

\begin{eqnarray*}
{K_r\bigg(\sum\limits_{h\in G}a_h\delta_h\bigg)}
& = & (r\delta_e)\bigg(\sum\limits_{h\in G}a_h\delta_h\bigg) -
 \bigg(\sum\limits_{h\in G}a_h\delta_h\bigg)(r\delta_e)\\
& = & \bigg(\sum\limits_{h\in G}r\alpha_e(a_h1_{e^{-1}})w_{e, h}\delta_{eh}\bigg) -
 \bigg(\sum\limits_{h\in G}a_h\alpha_h(r1_{h^{-1}})w_{h, e}\delta_{he}\bigg)\\
& = &\bigg(\sum\limits_{h\in G}ra_h1_h\delta_{h}\bigg) -
 \bigg(\sum\limits_{h\in G}a_h\alpha_h(r1_{h^{-1}})\delta_{h}\bigg)\\
& = &\bigg(\sum\limits_{h\in G}a_hr1_h\delta_{h}\bigg) -
 \bigg(\sum\limits_{h\in G}a_h\alpha_h(r1_{h^{-1}})\delta_{h}\bigg)\\
& = & \sum\limits_{h\in G}a_h\big(r1_h - \alpha_h(r1_{h^{-1}})\big)\delta_{h}\\
& = & \sum\limits_{h\in G\backslash\{e\}}a_h\big(r1_h - \alpha_h(r1_{h^{-1}})\big)\delta_{h}.
\end{eqnarray*}

Consequently, for each $r\in R$, the application $K_r$ always annihilates the
coefficient $\delta_e$ and it follows that $|K_r(\sum_{h\in G}a_h \delta_h)|<
|\sum_{h\in G}a_h\delta_h|$, for each  $0\neq \sum_{h\in G}a_h\delta_h$ with
$a_e\neq 0$.

By  assumption on $R$ and Corollary \ref{cor1.3} we have that
$C_{R*_{\alpha}^wG}(R) = \bigcap_{r\in R}\text{ker}(K_r)$.  For each
element $\sum_{h\in G}a_h\delta_h\in R*_{\alpha}^wG\backslash C_{R*_{\alpha}^wG}(R)$,
we choose  $r\in R$ such that  $\sum_{h\in G}a_h\delta_h\notin \text{ker}(K_r)$.
Thus, for each  $z=\sum_{g\in G}z_g \delta_g\in R*_{\alpha}^wG\backslash
C_{R*_{\alpha}^wG}(R)$ with $z_e\neq 0$, we choose $r\in R$ such that
$1 \leq |K_r(z)|<|\text{supp}(z)|$.

Finally, we are able to show that $I\cap C_{R*_{\alpha}^wG}(R)\neq 0$, for
each nonzero ideal  $I$ of $R*_{\alpha}^wG$. In fact, let $I$ be a nonzero
ideal of $R*_{\alpha}^wG$  and $0\neq z = \sum_{h\in G}a_h\delta_h\in I$.
If $z\in C_{R*_{\alpha}^wG}(R)$ the proof is complete. Now, suppose that
$x\notin C_{R*_{\alpha}^wG}(R)$. Then, applying  $T_g\,'s$ and $K_r\,'s$
in a suitable way until we obtain the nonzero element $b\delta_e\in I$ such that
$0\neq b\delta_e\in I\cap C_{R*^{w}_{\alpha}G}(R)$, since $T_g(I)\subseteq I$ and $K_r(I)\subseteq I$
\end{proof}

Using Lemma \ref{lem1.9} we immediately obtain the following result.

\begin{corollary}\label{cor1.10} Let $\alpha$ be a twisted partial action of a
group $G$ on a ring $R$. If  $R$ is maximal commutative in  $R*_{\alpha}^wG$,
then  $I\cap R\neq \{0\}$, for all nonzero ideal  $I$ of $R*_{\alpha}^wG$.
\end{corollary}

We recall that a ring  $S$ with a twisted partial action $\gamma$ of $G$ is said
to be $\gamma$-simple if the unique $\gamma$-invariant ideals of  $S$ are the
trivial ideals.

\begin{corollary}\label{cor1.11} Let $\alpha$ be a twisted partial action of a
group $G$ on a ring $R$. If $R$ is $\alpha$-simple and  maximal commutative in
$R*_{\alpha}^wG$, then  $R*_{\alpha}^wG$ is simple.
\end{corollary}
\begin{proof}
Let $I$ be a nonzero ideal of $R*_{\alpha}^wG$. Then $I\cap R$ is an $\alpha$-invariant
ideal of $R$. By assumption and  by Corollary \ref{cor1.10}, we have that
$I\cap R\neq \{0\}$.  Since  $R$  is  $\alpha$-simple, then $I\cap R = R$.
So, $R*_{\alpha}^wG$ is simple.
\end{proof}

The proof of the following lemma is standard and we left to the reader.

\begin{lemma}\label{lem1.12} Let $\alpha$ be a twisted partial action of a group
$G$ on a ring $R$. If $R*_{\alpha}^wG$ is simple, then   $R$ is $\alpha$-simple.
\end{lemma}

Using the Corollary  \ref{cor1.11} and Lemma \ref{lem1.12}, we obtain the first
principal result of this article, which generalizes (\cite{JO}, Theorem 6.13).

\begin{theorem}\label{teo1.13} Let $\alpha$ be a twisted partial action of a group
$G$ on $R$. Suppose that  $R$ is maximal commutative in $R*_{\alpha}^wG$. Then
$R*_{\alpha}^wG$ is simple if and only if $R$ is $\alpha$-simple.
\end{theorem}

\begin{lemma}\label{lem1.18}
Let $\alpha$ be a twisted partial action of a group $G$ on a ring $R$.
The center of $R*_{\alpha}^wG$ is
\begin{eqnarray*}
Z(R*_{\alpha}^wG) = \bigg\{\sum\limits_{g\in G}r_g\delta_g &|& r_{ts^{-1}}w_{ts^{-1}, s} =
\alpha_s(r_{s^{-1}t}1_{s^{-1}})w_{s, s^{-1}t},\\ & & r_s\alpha_s(a1_{s^{-1}}) = ar_s,
\forall a\in R\,\, and \,\, \forall s, t\in G\bigg\}.
\end{eqnarray*}
\end{lemma}
\begin{proof}
Let $\sum_{g\in G}r_g\delta_g \in Z(R*_{\alpha}^wG)$. Then,
$(\sum_{g\in G}r_g\delta_g) (a\delta_e)=(a\delta_e) (\sum_{g\in G} r_g\delta_g)$,
for any $a\in R$, and it follows that
\begin{eqnarray*}
\sum\limits_{g\in G}r_{g}\alpha_g(a1_{g^{-1}})\delta_{g}
&=& \sum\limits_{g\in G}r_{g}\alpha_g(a1_{g^{-1}})1_g\delta_{g}\\
&=& \sum\limits_{g\in G}r_{g}\alpha_g(a1_{g^{-1}})w_{g, e}\delta_{ge}\\
&=& \bigg(\sum\limits_{g\in G}r_g\delta_g\bigg)(a\delta_e)\\
&=& (a\delta_e)\bigg(\sum\limits_{g\in G}r_g\delta_g\bigg) \\
&=& \sum\limits_{g\in G}a\alpha_e(r_g1_{e^{-1}})w_{e, g}\delta_{eg}\\
&=& \sum\limits_{g\in G}ar_g1_g\delta_{g} = \sum\limits_{g\in
G}ar_g\delta_{g}.
\end{eqnarray*}
Thus, replacing $g$ by  $s$, we have that $r_s\alpha_s(a1_{s^{-1}}) = ar_s$,
$\forall a\in R$ and $\forall s\in G$. Since, $\sum_{g\in G}r_g\delta_g$ commutes
with $1_s\delta_s$, $\forall s\in G$, then
\begin{eqnarray*}
\sum\limits_{t\in G}r_{ts^{-1}}w_{ts^{-1}, s}\delta_t
&=& \sum\limits_{t\in G}r_{ts^{-1}}1_tw_{ts^{-1}, s}\delta_t =
 \sum\limits_{t\in G}r_{ts^{-1}}1_{ts^{-1}s}w_{ts^{-1}, s}\delta_t\\
&=& \sum\limits_{g\in G}r_{g}1_{gs}w_{g, s}\delta_{gs} =
 \sum\limits_{g\in G}r_{g}1_{g}1_{gs}w_{g, s}\delta_{gs}\\
&=& \sum\limits_{g\in G}r_{g}\alpha_g(1_{g^{-1}}1_{s})w_{g, s}\delta_{gs} =
 \sum\limits_{g\in G}r_{g}\alpha_g(1_{s}1_{g^{-1}})w_{g, s}\delta_{gs}\\
&=& \bigg(\sum\limits_{g\in G}r_g\delta_g\bigg)(1_s\delta_s) =
 (1_s\delta_s)\bigg(\sum\limits_{g\in G}r_g\delta_g\bigg)\\
&=& \sum\limits_{g\in G}1_s\alpha_s(r_g1_{s^{-1}})w_{s, g}\delta_{sg} =
 \sum\limits_{g\in G}\alpha_s(r_g1_{s^{-1}})w_{s, g}\delta_{sg}\\
&=& \sum\limits_{t\in G}\alpha_s(r_{s^{-1}t}1_{s^{-1}})w_{s, {s^{-1}t}}\delta_{t}.
\end{eqnarray*}
Hence, $r_{ts^{-1}}w_{ts^{-1}, s} = \alpha_s(r_{s^{-1}t}1_{s^{-1}})w_{s, s^{-1}t}$,
$\forall s, t\in G$,  and we have that
$Z(R*_{\alpha}^wG)\subseteq \{\sum_{g\in G}r_g\delta_g|  r_{ts^{-1}}w_{ts^{-1}, s} =
\alpha_s(r_{s^{-1}t}1_{s^{-1}})w_{s, s^{-1}t},  r_s\alpha_s(a1_{s^{-1}}) = ar_s,
\forall a\in R \text{ and } \forall s, t\in G\}$.

On the other hand,  let $\sum_{g\in G}r_g\delta_g$ be an element  of  $R*_{\alpha}^wG$
such that $r_s\alpha_s(a1_{s^{-1}}) = ar_s$ and $r_{ts^{-1}}w_{ts^{-1}, s} =
\alpha_s(r_{s^{-1}t}1_{s^{-1}})w_{s, s^{-1}t}$, $\forall a\in R$ and  $\forall s, t\in G$.
Then, for  any $\sum_{s\in G}a_s\delta_s\in R*_{\alpha}^wG$ we have that
\begin{eqnarray*}
\bigg(\sum\limits_{g\in G}r_g\delta_g\bigg)\bigg(\sum\limits_{s\in G}a_s\delta_s\bigg)
&=& \sum\limits_{s, g\in G}r_g\alpha_g(a_s1_{g^{-1}})w_{g, s}\delta_{gs}\\
&=& \sum\limits_{s, g\in G}a_sr_gw_{g, s}\delta_{gs}\\
&=& \sum\limits_{t, s\in G}a_sr_{ts^{-1}}w_{{ts^{-1}}, s}\delta_t\\
&=& \sum\limits_{t, s\in G}a_s\alpha_s(r_{s^{-1}t}1_{s^{-1}})w_{s, s^{-1}t}\delta_t\\
&=& \sum\limits_{g, s\in G}a_s\alpha_s(r_g1_{s^{-1}})w_{s, g}\delta_{sg}\\
&=& \bigg(\sum\limits_{s\in G}a_s\delta_s\bigg)\bigg(\sum\limits_{g\in G}r_g\delta_g\bigg).
\end{eqnarray*}
So, $\sum_{g\in G}r_g\delta_g$ commutes with any element of $R*_{\alpha}^wG$.
\end{proof}

In the next three corollaries we obtain a description of the center of partial
crossed product when we input some other assumptions either on $R$ or on the
twisted partial action $\alpha$.

\begin{corollary}\label{cor1.19}
Let $\alpha$ be a twisted partial action of a group $G$ on a ring $R$.
If $\alpha_g = id_{D_g}$, $\forall g\in G$, then
\begin{eqnarray*}
Z(R*_{\alpha}^wG) = \bigg\{\sum\limits_{g\in G}r_g\delta_g &|& r_s\in Z(R),
r_{ts^{-1}}w_{ts^{-1}, s} = r_{s^{-1}t}w_{s, s^{-1}t}, \forall s, t\in G\bigg\}.
\end{eqnarray*}
\end{corollary}
\begin{proof}
By assumption, we have that $\alpha_g(r_g) = r_g$, $\forall r_g\in D_g$, and
$1_{g^{-1}} = 1_g$, $\forall g\in G$.  Thus,
\begin{eqnarray}\label{ig1}
\alpha_s(r_{s^{-1}t}1_{s^{-1}})w_{s, s^{-1}t} = r_{s^{-1}t}1_sw_{s, s^{-1}t}, \forall s, t\in G;
\end{eqnarray}
\begin{eqnarray}\label{ig2}
\alpha_s(a1_{s^{-1}}) = a1_s, \forall a\in R \text{ and } \forall s\in G.
\end{eqnarray}
Since $w_{s, s^{-1}t}\in D_sD_{ss^{-1}t} = D_sD_t\subseteq D_s$, we have
$1_sw_{s, s^{-1}t} = w_{s, s^{-1}t}$.  Using the Lemma \ref{lem1.18}
and the equality (\ref{ig1}),  we obtain that
\begin{center}
$r_{ts^{-1}}w_{ts^{-1}, s} = r_{s^{-1}t}w_{s, s^{-1}t}, \forall s, t\in G$.
\end{center}
By the fact that $r_sa\in D_s$, $\forall a\in R$, we have $r_sa1_s = r_sa$.
Hence, by Lemma \ref{lem1.18} and the equality (\ref{ig2}), we obtain
that $r_sa = ar_s$, $\forall a\in R$, i.e.  $r_s\in Z(R)$.
\end{proof}

Let $\alpha$ be a twisted partial action  of a group $G$ on a ring $R$. The subring
of partial invariants is defined by
$R^{\alpha}=\{ r\in R: \alpha_g(r1_{g^{-1}})=r1_g, \forall g\in G\}$, see \cite{PA}.

\begin{corollary}\label{cor1.20}
Let $\alpha$ be a twisted partial action of a group $G$ on a ring $R$ and suppose
that $G$ is abelian and  $w$ is symmetric. Then
\begin{eqnarray*}
Z(R*_{\alpha}^wG) = \bigg\{\sum\limits_{g\in G}r_g\delta_g &|& r_s\in R^{\alpha},
r_s\alpha_s(a1_{s^{-1}}) = ar_s, \forall a\in R \,\, and \,\, \forall  s\in G\bigg\}.
\end{eqnarray*}
\end{corollary}
\begin{proof}
By assumption we have that $w_{ts^{-1}, s} = w_{s^{-1}t, s} = w_{s, s^{-1}t}$,
$\forall s, t\in G$. Since $w_{s, s^{-1}t}\in D_sD_{ss^{-1}t} = D_sD_t\subseteq D_s$,
we have $w_{s, s^{-1}t} = 1_sw_{s, s^{-1}t}$. Thus, by Lemma \ref{lem1.18}, it
follows that  $\alpha_s(r_{s^{-1}t}1_{s^{-1}})w_{s, s^{-1}t} =
r_{s^{-1}t}1_sw_{s, s^{-1}t}$, $\forall s, t\in G$. Hence,
$\alpha_s(r_{s^{-1}t}1_{s^{-1}}) = r_{s^{-1}t}1_s$, $\forall s, t\in G$.
So, $r_s\in R^{\alpha}$, $\forall s\in G$.
\end{proof}

\begin{corollary}\label{cor1.21}
Let $\alpha$ be a twisted partial action of a group $G$ on a ring $R$ and suppose
that $G$ is an abelian. If one of the following conditions are satisfied
\begin{itemize}
\item [$(i)$] $R$ is  commutative and $w \equiv 1_g1_{gh}$, $\forall g,
h\in G$;
\item [$(ii)$] $R$ is commutative and $w$ is symmetric;
\item [$(iii)$] $w$ is symmetric and $\alpha_g = id_{D_g}$, $\forall g\in G$,
\end{itemize}
then
\begin{eqnarray*}
Z(R*_{\alpha}^wG) = \bigg\{\sum\limits_{g\in G}r_g\delta_g &|& r_s\in R^{\alpha},
\alpha_s(a1_{g^{-1}}) - a \in \text{\emph{ann}}(r_s), \forall a\in R \text{ and }
\forall s\in G\bigg\}.
\end{eqnarray*}
\end{corollary}
\begin{proof}
Suppose that  (\emph{i}) holds. Since $w \equiv 1_g1_{gh}$, $\forall g, h\in G$, it
follows that $\sum_{g\in G}r_g\delta_g \in Z(R*_{\alpha}^wG)$ if and only if
$r_g\in R^{\alpha}$, $g\in G$. By Lemma \ref{lem1.18}, and the fact that $R$
is commutative we have that $r_s\alpha_s(a1_{s^{-1}}) = ar_s = r_sa$. Thus,
$\alpha_s(a1_{g^{-1}}) - a \in \text{ann}(r_s)$, $\forall a\in R$ and $\forall s\in G$.

Suppose that (\emph{ii}) holds. By Corollary \ref{cor1.20} we have that $r_s\in R^{\alpha}$,
$\forall s\in G$, and by similar methods as above we obtain that
$\alpha_s(a1_{g^{-1}}) - a \in \text{ann}(r_s)$, $\forall a\in R$.

Finally, suppose that (\emph{iii}) holds. By Corollary \ref{cor1.20} we have that
$r_s\in R^{\alpha}$. Since $\alpha_g = id_{D_g}$, $\forall g\in G$, by Corollary
\ref{cor1.19} we have that $r_s\in Z(R)$, $\forall s\in G$. Moreover, by similar
methods as before, we obtain that $\alpha_s(a1_{g^{-1}}) - a \in \text{ann}(r_s)$,
$\forall a\in R$ and $\forall s\in G$.
\end{proof}

We need the following result to show when $R*^{w}_{\alpha}G$ is commutative.

\begin{lemma}\label{lem1.22}
Let $\alpha$ be a twisted partial action of a group $G$ on a ring $R$.
If  $R*_{\alpha}^wG$ is commutative, then   $w_{g, g^{-1}} = w_{g^{-1}, g}$,
$\forall g\in G$.
\end{lemma}
\begin{proof}
By the fact that $w_{g, g^{-1}}\in D_gD_{gg^{-1}} = D_gD_e = D_g$ and
$w_{g^{-1}, g}\in D_{g^{-1}}D_{g^{-1}g} = D_{g^{-1}}D_e = D_{g^{-1}}$,
$\forall g\in G$, we have that $1_gw_{g, g^{-1}} = w_{g, g^{-1}}$ and
$1_{g^{-1}}w_{g^{-1}, g} = w_{g^{-1}, g}$. Since
$(1_g\delta_g)(1_{g^{-1}}\delta_{g^{-1}}) = (1_{g^{-1}}\delta_{g^{-1}})(1_g\delta_g)$,
$\forall g\in G$, it follows that
\begin{eqnarray*}
w_{g, g^{-1}}\delta_e
&=& 1_gw_{g, g^{-1}}\delta_e = \alpha_g(1_{g^{-1}})w_{g, g^{-1}}\delta_{gg^{-1}}\\
&=& (1_g\delta_g)(1_{g^{-1}}\delta_{g^{-1}}) = (1_{g^{-1}}\delta_{g^{-1}})(1_g\delta_g)\\
&=& \alpha_{g^{-1}}(1_g)w_{g^{-1}, g}\delta_{g^{-1}g} = 1_{g^{-1}}w_{g^{-1}, g}\delta_{e}\\
&=& w_{g^{-1}, g}\delta_{e}.
\end{eqnarray*}
 So, $w_{g, g^{-1}} = w_{g^{-1}, g}$, $\forall g\in G$.
\end{proof}

The next result  provides necessary and sufficient conditions for the commutativity
of the partial crossed product which generalizes (\cite{JO}, Corollary 4) and
partially generalizes (\cite{CFL}, Proposition 2).

\begin{theorem}\label{cor1.24}
Let $\alpha$ be a twisted partial action of $G$ on $R$.  The partial crossed product
$R*_{\alpha}^wG$  is commutative  if and only if $R$ is  commutative, $G$ is abelian,
$w$ is symmetric and $\alpha_g = id_{D_g}$, $\forall g\in G$.
\end{theorem}
\begin{proof}
Suppose that  $R*_{\alpha}^wG$ is commutative. Then, in particular, $R$ is commutative.

We show that $D_g = D_{g^{-1}}$, $\forall g\in G$. In fact, since $w_{g, g^{-1}}$ is
invertible in $D_gD_{gg^{-1}} = D_gD_e = D_g$, there exists  $w^{-1}_{g, g^{-1}}\in D_g$
such that $w_{g, g^{-1}}\cdot w^{-1}_{g, g^{-1}} = 1_g$. By Lemma \ref{lem1.22},
$w_{g, g^{-1}} = w_{g^{-1}, g}$, $\forall g\in G$ and since $w_{g^{-1}, g}\in D_{g^{-1}}$
it follows that $1_g\in D_{g^{-1}}$. Thus, $D_g\subseteq D_{g^{-1}}$ and by similar
methods we obtain the other inclusion. Consequently, $1_g = 1_{g^{-1}}$, $\forall g\in G$.
By  the fact that $1_g\delta_g\in Z(R*_{\alpha}^wG)$ and using Lemma \ref{lem1.18},
we have that $1_g\alpha_g(a1_{g^{-1}}) = a1_g$, $\forall a\in R$. So, for any
$a\in D_g = D_{g^{-1}}$, we have $\alpha_g(a) = a$ and therefore $\alpha_g = id_{D_g}$,
$\forall g\in G$.

We claim that $D_gD_{gh} = D_hD_{hg}$, $\forall g, h \in G$. In fact, since
$\alpha_g(D_{g^{-1}}D_h) = D_gD_{gh}$, $\alpha_h(D_{h^{-1}}D_g) = D_hD_{hg}$,
$D_gD_h = D_hD_g$, $D_{g^{-1}}D_h\subseteq D_{g^{-1}}$, $D_{h^{-1}}D_g\subseteq D_{h^{-1}}$,
$D_g = D_{g^{-1}}$ and $\alpha_g = id_{D_g}$, $\forall g,h\in G$, we have that
$D_gD_{gh} = \alpha_g(D_{g^{-1}}D_h) = D_{g^{-1}}D_h = D_gD_{h^{-1}} = D_{h^{-1}}D_g =
\alpha_h(D_{h^{-1}}D_g) = D_hD_{hg}$. Moreover, since we have
$(1_g\delta_g)(1_h\delta_h) = (1_h\delta_h)(1_g\delta_g)$, $\forall g, h\in G$,
it follows that
\begin{eqnarray}\label{eq3}
w_{g, h}\delta_{gh}\!\! &=& \!\!\! 1_h1_{g^{-1}}w_{g, h}\delta_{gh} =
\alpha_g(1_h1_{g^{-1}})w_{g, h}\delta_{gh} \nonumber \\
&=& (1_g\delta_g)(1_h\delta_h) = (1_h\delta_h)(1_g\delta_g) \\
&=& \alpha_h(1_g1_{h^{-1}})w_{h, g}\delta_{hg} = 1_g1_{h^{-1}}w_{h, g}\delta_{hg} =
w_{h, g}\delta_{hg}. \nonumber
\end{eqnarray}
Using  the fact that $w_{g, h}\in U(D_gD_{gh})\subseteq D_{gh}$ and
$w_{h, g}\in  U(D_hD_{hg})\subseteq D_{hg}$, from the equality (\ref{eq3}),
we have $gh = hg$, $\forall g, h \in G$, and consequently  $G$ is abelian.
Also, from the equality (\ref{eq3}), we obtain that $w_{g, h} = w_{h, g}$,
$\forall g, h \in G$, and since $D_gD_{gh} = D_hD_{hg}$, $\forall g, h \in G$,
it follows that $w$ is symmetric.

Conversely, suppose that $R$ is commutative, $\alpha_g = id_{D_g}$, $\forall g\in G$,
$G$ is abelian and $w$ is  symmetric. Let $\sum_{g\in G}a_g\delta_g$ and
$\sum_{h\in G}b_h\delta_h$ be elements of $R*_{\alpha}^wG$. Then,
\begin{eqnarray*}
\bigg(\sum\limits_{g\in G}a_g\delta_g\bigg)\bigg(\sum\limits_{h\in G}b_h\delta_h\bigg) \!\!
&=&\!\!\! \sum\limits_{g, h\in G}a_g\alpha_g(b_h1_{g^{-1}})w_{g, h}\delta_{gh} =
 \sum\limits_{g, h\in G}a_gb_h1_{g^{-1}}w_{g, h}\delta_{gh}\\
&=& \sum\limits_{g, h\in G}a_gb_h1_gw_{g, h}\delta_{gh} =
 \sum\limits_{g, h\in G}a_gb_hw_{g, h}\delta_{gh}\\
&=& \sum\limits_{g, h\in G}a_gb_hw_{h, g}\delta_{gh} =
 \sum\limits_{g, h\in G}a_gb_hw_{h, g}\delta_{hg}\\
&=& \sum\limits_{g, h\in G}b_ha_g1_hw_{h, g}\delta_{hg} =
 \sum\limits_{g, h\in G}b_ha_g1_{h^{-1}}w_{h, g}\delta_{hg}\\
&=& \sum\limits_{g, h\in G}b_h\alpha_h(a_g1_{h^{-1}})w_{h, g}\delta_{hg} =
 \bigg(\sum\limits_{h\in G}b_h\delta_h\bigg)\bigg(\sum\limits_{g\in G}a_g\delta_g\bigg).
\end{eqnarray*}
and it follows that  $R*_{\alpha}^wG$ is  commutative.
\end{proof}

\begin{lemma}\label{lem1.29}
Let $\alpha$ be a twisted partial action of a group $G$ on a ring $R$. Then
for every nonzero ideal $I$ of $R*_{\alpha}^wG$ we have that
$I\cap C_{R*_{\alpha}^wG}(Z(R))\neq \{0\}$.
\end{lemma}
\begin{proof}
It  is  enough to show that if $I\cap C_{R*_{\alpha}^wG}(Z(R))=\{0\}$, then $I=(0)$.
Let  $x = \sum_{h\in G}a_h\delta_h\in I$. If $x\in C_{R*_{\alpha}^wG}(Z(R))$,
then $x = 0$ by assumption. Thus, we assume that there exists
$z\in I\backslash C_{R*_{\alpha}^wG}(Z(R))$ and we choice $x\in I\backslash C_{R*_{\alpha}^wG}(Z(R))$  among the elements of
$ I\backslash C_{R*_{\alpha}^wG}(Z(R))$ such that $|\text{supp}(x)|$
is minimal. Note that, for any $p\in \text{supp}(x)$, we have
$x' = x1_{p^{-1}}\delta_{p^{-1}}\in I\backslash C_{R*_{\alpha}^{w}G}(Z(R))$,
with $e\in \text{supp}(x')$ and $|\text{supp}(x')| = |\text{supp}(x)|$. Hence,
we may assume that $e\in \text{supp}(x)$. For each  $r\in Z(R)$, let $x'' = rx - xr$.
By the fact that  $r\in Z(R)$ and $r1_e - \alpha_e(r1_{e^{-1}}) = 0$, we have that
\begin{eqnarray*}
x'' & = & (r\delta_e)\bigg(\sum\limits_{h\in G}a_h\delta_h\bigg) -
 \bigg(\sum\limits_{h\in G}a_h\delta_h\bigg)(r\delta_e)\\
& = & \bigg(\sum\limits_{h\in G}r\alpha_e(a_h1_{e^{-1}})w_{e, h}\delta_{eh}\bigg) -
 \bigg(\sum\limits_{h\in G}a_h\alpha_h(r1_{h^{-1}})w_{h, e}\delta_{he}\bigg)\\
& = &\bigg(\sum\limits_{h\in G}ra_h1_h\delta_{h}\bigg) -
 \bigg(\sum\limits_{h\in G}a_h\alpha_h(r1_{h^{-1}})\delta_{h}\bigg)\\
& = &\bigg(\sum\limits_{h\in G}a_hr1_h\delta_{h}\bigg) -
 \bigg(\sum\limits_{h\in G}a_h\alpha_h(r1_{h^{-1}})\delta_{h}\bigg)\\
& = & \sum\limits_{h\in G}a_h\big(r1_h - \alpha_h(r1_{h^{-1}})\big)\delta_{h}\\
& = & \sum\limits_{h\in G\backslash\{e\}}a_h\big(r1_h - \alpha_h(r1_{h^{-1}})\big)\delta_{h}.
\end{eqnarray*}
Consequently,  $e\notin \text{supp}(x'')$  and it follows that $|\text{supp}(x'')| < |\text{supp}(x)|$.
Since $x''\in I$, by the minimality of $|\text{supp}(x)|$, we obtain that
$rx = xr$, $\forall r\in Z(R)$. So, $x\in C_{R*_{\alpha}^wG}(Z(R))$ which
contradicts the choice of $x$. Therefore, $I= (0)$.

The proof is complete.
\end{proof}

We are in conditions to prove the second principal  result of this section.

\begin{theorem}\label{teo1.30}
Suppose that  $C_{R*_{\alpha}^wG}(Z(R))$ is a simple ring. Then $R*_{\alpha}^wG$ is
simple if and only if $R$ is $\alpha$-simple.
\end{theorem}
\begin{proof}
If  $R*_{\alpha}^wG$ is simple, then,  by Lemma \ref{lem1.12}, $R$ is $\alpha$-simple.

Conversely, suppose that $R$ is $\alpha$-simple and let $I$ be a nonzero ideal of
$R*_{\alpha}^wG$. If  $I\cap R\neq \{0\}$ the result follows. Suppose that $I\cap R=\{0\}$.
Thus, by Lemma \ref{lem1.29}, we have that $I\cap C_{R*_{\alpha}^wG}(Z(R))\neq \{0\}$.
By the fact that $C_{R*_{\alpha}^wG}(Z(R))$ is simple we have that  $1_R\in I$ and
it follows that  $I = R*_{\alpha}^wG$. So, $R*_{\alpha}^wG$ is simple.
\end{proof}

In the next theorem,  we simply denote $\delta_g$ by $g$  and  $e = g_1$.

\begin{theorem}\label{teo1.31}
Suppose that $G$ is abelian and $C_{R*_{\alpha}^wG}(R)$ is simple.
Then $R*_{\alpha}^wG$ is simple if and only if $R$ is $\alpha$-simple.
\end{theorem}
\begin{proof}
If  $R*_{\alpha}^wG$ is simple then , by Lemma \ref{lem1.12}, $R$ is $\alpha$-simple.

Conversely, suppose that $R$ is $\alpha$-simple and that $I$  is a nonzero ideal
of $R*_{\alpha}^wG$. If $I\cap R\neq \{0\}$ we are ready. Now, suppose that $I\cap R=\{0\}$
and let $x = \sum_{i = 1}^{n} a_ig_i\in I$ such that  $|\text{supp}(x)|$ is minimal.
As before, we may assume that  $e = g_1\in \text{supp}(x)$. Since $R$ is
$\alpha$-simple, then $\big\{\alpha_g\big(a1_{g^{-1}}\big)|g\in G\big\}$ generates
$R$ and it follows that there exists $r_{kj}, s_{kj}\in R$ and $g_j\in G$ such that
$1 = \sum\limits_{j}\sum\limits_{k}r_{kj}\alpha_{g_j}\big(a_11_{{g_j}^{-1}}\big)s_{kj}$.
Let
$$y = \sum\limits_{j}\sum\limits_{k}r_{kj}1_{g_j}g_jx1_{{g_j}^{-1}}{g_j}^{\!-1}w^{-1}_{g_j, {g_j}^{\!-1}}s_{kj}.$$
Then we have that
\begin{eqnarray*}
y & = &\sum\limits_{j}\sum\limits_{k}r_{kj}\alpha_{g_j}\big(a_11_{{g_j}^{-1}}\big)s_{kj} + \\
& & + \sum\limits_{j}\sum\limits_{k}r_{kj}\bigg(\sum\limits_{i = 2}\limits^{n}
 \alpha_{g_j}\big(a_i1_{{g_j}^{-1}}\big)w_{g_j, g_i}w_{g_j g_i, {g_j}^{-1}}\alpha_{g_i}\Big(w^{-1}_{g_j, {g_j}^{\!-1}}1_{{g_i}^{-1}}\Big)g_i\bigg)s_{kj}\\
& = & 1 + \sum\limits_{j}\sum\limits_{k}r_{kj}\bigg(\sum\limits_{i = 2}\limits^{n}
 \alpha_{g_j}\big(a_i1_{{g_j}^{-1}}\big)w_{g_j, g_i}w_{g_j g_i, {g_j}^{-1}}\alpha_{g_i}\Big(w^{-1}_{g_j, {g_j}^{\!-1}}1_{{g_i}^{-1}}\Big)g_i\bigg)s_{kj}\\
& = & 1 + \sum\limits_{i = 2}\limits^{n}\bigg(\sum\limits_{j}\sum\limits_{k}r_{kj}\alpha_{g_j}
 \big(a_i1_{{g_j}^{-1}}\big)w_{g_j, g_i}w_{g_j g_i, g_{j^{-1}}}\alpha_{g_i}\Big(w^{-1}_{g_j,
 {g_j}^{\!-1}}s_{kj}1_{{g_i}^{-1}}\Big)\bigg)g_i.
\end{eqnarray*}
Thus, $y\in I$ is such that  $|\text{supp}(y)|$ is minimal and we consider
$x = 1 + \sum_{i = 2}^{n} a_ig_i$.

For each $r\in R$ we have that the element $rx - xr$ satisfies
$|\text{supp}(rx - xr)| < |\text{supp}(x)|$. By the fact that $|\text{supp}(x)|$
is minimal and $rx - xr\in I$, we have that $rx = xr$, $\forall r\in R$, i.e
$x\in C_{R*_{\alpha}^wG}(R)$. Hence, $x\in I\cap C_{R*_{\alpha}^wG}(R)$ and since
$C_{R*_{\alpha}^wG}(R)$ is simple we have that  $1_R\in I$. So, $I = R*_{\alpha}^wG$
and consequently  $R*_{\alpha}^wG$ is simple.
\end{proof}

We define $E:R*_{\alpha}^{w}G\rightarrow R$ by $E(\sum_{g\in G}a_g\delta_g)=a_e$.
We clearly have that $E$ is an homomorphism of left $R$-modules.

\begin{lemma}\label{lem1.32}
Suppose that $R$ is $\alpha$-simple. Then, for each $y\in A= R*_{\alpha}^wG$,
there exists $y'\in R*_{\alpha}^wG$ such that the following properties hold:
\begin{itemize}
\item [$(i)$] $y'\in AyA$;
\item [$(ii)$] $E(y') = 1$;
\item [$(iii)$] $|\emph{supp}(y')| \meig |\emph{supp}(y)|$.
\end{itemize}
\end{lemma}
\begin{proof}
Let  $y=\sum_{g\in G}a_g\delta _g$ be a nonzero element of $A = R*_{\alpha}^wG$.
Then there exists  $h\in G$ such that $a_h\dif 0$. Note that
$y1_{h^{-1}}\delta_{h^{-1}}\in AyA$,  $|\text{supp}(y1_{h^{-1}}\delta_{h^{-1}})|\leq
|\text{supp}(y)|$ and $E(y1_{h^{-1}}\delta_{h^{-1}}) = a_h\dif 0$. Thus, without
loss of generality, we replace $y$ by $y1_{h^{-1}}\delta_{h^{-1}}$ and we assume
that $e\in \text{supp}(y)$.

The set  $J = \{E(s)|s\in AyA \text{ such that } |\text{supp}(s)| \leq
|\text{supp}(y)|\}$ contains $a_e$ because $y\in AyA$ and it follows that
$J$ is a nonzero ideal of  $R$. We claim that $J$ is $\alpha$-invariant.
In fact, let   $a\in J\cap D_{h^{-1}}$. Then  there exists an element $s
= a + \sum_{g\in \text{supp}(y)\backslash\{e\}}b_g\delta_g\in AyA$. Since
$a\in D_{h^{-1}}$ and $w_{h, e} = 1_h$, we have that
$1_{h}\delta_ha1_{h^{-1}}\delta_{h^{-1}} = \alpha_h(a)$. By assumption on
$G$ we have that $1_h\delta_h(b_g\delta_g)1_{h^{-1}}\delta_{h^{-1}} =
\alpha_h(b_g1_{h^{-1}})w_{h, g}\alpha_{hg}(1_{h^{-1}}1_{hg^{-1}})w_{hg, h^{-1}} \delta_g$,
$\forall g, h\in G$.  Thus, we have that
$$\alpha_h(a) + \sum_{g\in G}[\alpha_h(b_g1_{h^{-1}})w_{h, g}\alpha_{hg}
(1_{h^{-1}}1_{hg^{-1}})w_{hg, h^{-1}}]\delta_{hgh^{-1}}=$$
$$=1_{h}\delta_h(a + \sum b_g\delta_g)1_{h^{-1}}\delta_{h^{-1}}\in AyA.$$
with $|\text{supp}(1_{h}\delta_h(a + \sum b_g\delta_g)1_{h^{-1}}\delta_{h^{-1}})
|\leq |\text{supp}(s)| \leq |\text{supp}(y)|$ and consequently $\alpha_h(a)\in J$.
Hence, $\alpha_h(J\cap D_{h{-1}})\cont J\cap D_h$, $\forall h\in G$, and it follows that
$J$ is $\alpha$-invariant. By the fact that $R$ is $\alpha$-simple, we have that $J = R$
and there exists $y' := 1 + \sum_{g\in \text{supp}(y)\backslash\{e\}}b_g\delta_g\in AyA$,
that is the required element.
\end{proof}
\hskip-5mm%

\begin{lemma}\label{lem1.33}
If $R$ is $\alpha$-simple and $G$ is abelian, then all nonzero ideals $I$ of
$R*_{\alpha}^wG$ satisfies $I\cap Z(R*_{\alpha}^wG)\neq \{0\}$.
\end{lemma}
\begin{proof}
Let  $I$ be a nonzero ideal of  $A = R*_{\alpha}^wG$ and  $y$ a nonzero element
of $I$ such that  $|\text{supp}(y)|$ is minimal. By Lemma \ref{lem1.32}, there
exists $y'\in AyA\cont I$ such that  $E(y') = 1$ and
$|\text{supp}(y')| \meig |\text{supp}(y)|$. Thus, by assumption on $y$,
we have that  $|\text{supp}(y')| = |\text{supp}(y)|$.

For each  $s\in R$ we have that $E(y's - sy') = s - s = 0$ and it follows that
$|\text{supp}(y's - sy')| < |\text{supp}(y')|$. By minimality of $|\text{supp}(y')|$
and since $y's - sy'\in I$, we obtain that $y's = sy'$. By the fact that
$(1_g\delta_g)(1\delta_e)=(1\delta_e)(1_g\delta_g)$, it follows that
$|\text{supp}((1_g\delta_g)y' - y'(1_g\delta_g))|\leq |\text{supp}(y')|$ and
since $(1_g\delta_g)y' - y'(1_g\delta_g)\in I$,  we obtain that
$(1_g\delta_g)y' = y'(1_g\delta_g)$, $\forall g\in G$. Hence,
$y'\in Z(R*_{\alpha}^wG)$ and consequently, $y'\in I\cap Z(R*_{\alpha}^wG)$.
So, $I\cap Z(R*_{\alpha}^wG) \neq \{0\}$, because $y'\dif 0$ since $E(y') = 1$.
\end{proof}

Now, we are ready to prove the last result of this section which partially generalizes
(\cite{JO1}, Theorem 1.2).

\begin{theorem}\label{teo1.34}
Suppose that $G$ is an abelian group. Then $R*_{\alpha}^wG$ is simple if
and only if  $Z(R*_{\alpha}^wG)$ is a field and $R$ is $\alpha$-simple.
\end{theorem}
\begin{proof}
Suppose that  $R*_{\alpha}^wG$ is simple. Thus, for each
$0\neq x\in Z(R*_{\alpha}^wG)$, we have that
$(R*_{\alpha}^wG)x=x(R*_{\alpha}^wG) = R*_{\alpha}^wG$ and it follows that
there exists $x^{-1}\in R*_{\alpha}^wG$ such that  $xx^{-1} = x^{-1}x = 1$.
Note that for any $a\in R*_{\alpha}^wG$  we obtain that
$x(x^{-1}a) = (xx^{-1})a = a(xx^{-1}) = (ax)x^{-1} =
(xa)x^{-1} = x(ax^{-1})$, i.e, $x(x^{-1}a) = x(ax^{-1})$.
Hence, $x^{-1}a = ax^{-1}$, for any $a\in R*_{\alpha}^wG$, and we have
that  $x^{-1}\in Z(R*_{\alpha}^wG)$. So, $x$ is invertible in
$Z(R*_{\alpha}^wG)$  and we have that  $Z(R*_{\alpha}^wG)$ is a field.
Moreover, by Lemma \ref{lem1.12}, $R$ is $\alpha$-simple.

Conversely, suppose that   $Z(R*_{\alpha}^wG)$ is a field and  $R$ is
$\alpha$-simple. Let $I$ be a nonzero ideal of $R*_{\alpha}^wG$. Then,
by Lemma \ref{lem1.33}, we have that $I\cap Z(R*_{\alpha}^wG)\neq \{0\}$.
So, $I=R*^{w}_{\alpha}G$ and we have that  $R*^{w}_{\alpha}G$ is  simple.
\end{proof}

Now, we finish  with the following example where we apply the results of this
section to conclude that the partial crossed product is not simple. Moreover, 
it shows that the assumptions on Theorem \ref{teo1.34} are not superfluous.

\begin{example}
\emph{Let $T=Ke_1\oplus Ke_2\oplus Ke_3$, where $K$ is a field and $\{e_1, e_2,e_3\}$
are central orthogonal idempotents, and $R=Ke_1$. We define the action of $\mathbb{Z}$
on $T$ as follows: $\sigma(e_1)=e_2$, $\sigma(e_2)=e_3$ and $\sigma(e_3)=e_1$. We have
the following induced partial action of $\mathbb{Z}$ on $R$: $D_j=R$, for
$j\equiv 0 (mod\, 3)$, and $D_k=D_l=(0)$, for $k\equiv 1(mod\, 3)$ and $l\equiv 2(mod\, 3)$,
with isomorphisms $\alpha_j=id_R$, for $j\equiv 0(mod\, 3)$, and $\alpha_k=\alpha_l=0$,
for $k\equiv 1(mod\, 3)$ and $l\equiv 2(mod\, 3)$. By Theorem \ref{cor1.24},
$R*_{\alpha}G$ is commutative. Note that $R*_{\alpha}G$ is not field and by Theorem
\ref{teo1.34}, $R*_{\alpha}G$ is not simple.}
\end{example}

\section{Applications}

In this section, we study some topological properties and  applications of the last section 
to the $C^{*}$-algebra of type $C(X)$, where $X$ is a topological space.

\subsection{Some properties of partial dynamical systems}

Given a partial dynamical system $(X, \al, G)$ the \emph{partial orbit} of a
point $x \in X$  is the set
$$
    O^{\alpha}(x) = \{ \al_t(x): x \in X_{t^{-1}} \text{, } t\in G \}
$$

A partial dynamical system is said to be {\em transitive} if there exists some
$x_0 \in X$ such that $O^{\al}(x)$ is dense in $X$, i.e. if $\overline{O^{\al}(x_0)} = X$.
If for every $x \in X$,  $O^{\al}(x)$ is dense in $X$, we say that the partial dynamical
system is {\em minimal}.

 We remind that a topological space $X$ is \emph{compact} if, for every
collection $\{U_{i}\}_{i\in I}$ of open sets in $X$ whose union is $X$, there
exists a finite sub-collection $\{U_{i_j}\}^{n}_{j=1}$ whose union is also $X$.

In the next result we assume that $X$ is a compact metric space and we can state a condition
that implies transitivity.

\begin{theorem}\label{criteriotransitivo}
Let $(X, \al, G)$ be a partial dynamical system such that $X$ is a compact
metric space. Then the following conditions are equivalent:
\begin{itemize}
\item[$(i)$] $(X, \al, G)$ is transitive.
\item[$(ii)$] Given any two open sets $U$ and $V$ in $X$,  there exists some $g \in G$
 such that  $\al_g(U \cap X_{g^{-1}}) \cap V \neq \emptyset$.
\end{itemize}
\end{theorem}
\begin{proof}
Since $(i) \Rightarrow (ii)$ is clear from the definition, then  we concentrate on
the proof of the other implication. We want to show that for any real number
$\omega > 0$, there exists an orbit that is $\omega-$dense, i.e. such that any
point of $X$ is at a distance smaller than $\omega$ from the orbit. For this
purpose, take a covering of $X$ by open balls of radius $\omega$. Thus, by
compacity, we extract a finite subcovering $B_1, \ldots, B_N$ such that
$X \subset B_1 \cup \cdots \cup B_N$. Hence, by assumption, there exists some
$t_1 \in G$ such that $\al_{t_1}(B_1 \cap X_{t_1^{-1}}) \cap B_2 \neq \emptyset$ and
this implies that $\al_{t_1^{-1}}(B_2 \cap X_{t_1})\cap B_1 =: B_{12} \neq \emptyset$,
which is an open set. Note that by assumption, this set has some image
intersecting the open set $B_3$, because of this we have some $t_2 \in G$ such that
$\al_{t_2}(B_{12} \cap X_{t_2^{-1}}) \cap B_3 \neq \emptyset$. Now we have the open
set $\al_{t_2^{-1}}(B_3 \cap X_{t_2}) \cap B_{12} =: B_{123} \neq \emptyset$. Repeating
this procedure $n$ times we get an open set $B_{12 \cdots n}$ such that any point in it
has images in $B_1$, $B_2$, $\ldots$, $B_n$, being $\omega-$dense, as desired.

Since we have $\omega$-dense orbits for any $\omega > 0$ then we have in fact
dense orbits for this partial dynamical system.
\end{proof}

Now, we have the following result.

\begin{proposition}\label{dense1}
Let $(X, \alpha, G)$ be a partial dynamical system, $X$ compact, $(X^e, \beta, G)$ its
enveloping action and $x_0 \in X$. The following statements hold:
\begin{itemize}
\item[$(i)$] Suppose that $\overline{O^{\alpha}(x_0)} = X$, then
$\overline{\beta_g([e, O^{\alpha}(x_0)])} = \beta_g([e, X])$
 for every $g \in G$.
\item[$(ii)$] $(X, \alpha, G)$ is transitive if and only if $(X^e, \beta, G)$
 is transitive.
\end{itemize}
\end{proposition}
\begin{proof}
(\emph{i}) We clearly have $\be_g ([e, O^{\alpha}(x_0) ])  \subseteq \be_g ([e, X])$.
We claim that $\beta_g([e, X])$ is closed. In fact,
$$
      \beta_g([e, X]) = [g, X] = \bigcup_{x \in X} [g, x]
$$
If $[g, x_n] \to [g, x]$ then we have $x_n \to x$ and since $X$ is closed, we
have that $x \in X$. Hence,  $[g, x] \in [g, X]= \beta_g([e, X])$.

So, $\overline{ \beta_g ([e, O^{\alpha}(x_0) ])} = \beta_g ([e, X])$.

(\emph{ii}) Suppose that $(X, \al, G)$ is transitive. Then,
given $x \in X$ there exists $x_0\in X$ and  a sequence
$\{ g_n \}_{n \in \mathbb{N}} \subset G$ such that $\al_{g_n}(x_0) \to x$.
Now,  take a point  $[g, x] \in X^e$. Then we have
$$
     \beta_{g g_n}([e, x_0])= \beta_g([g_n, x_0]) = \beta_g([e, \al_{g_n}(x_0)])=
       [g, \al_{g_n}(x_0)] \to [g, x].
$$
So, $(X^e, \beta, G)$ is also transitive.

Conversely, if $(X^e, \beta, G)$ is transitive, then there exists $[g_0 , x_0]$
such that for any $[g, x]$ we can find $\{g_n\}_{n\in \mathbb{N}}\subset G$ with
the property $\beta_{g_n}([g_0, x_0]) \to [g, x]$. In particular, for $g = e$ we have
$$
    \beta_{g_n}([g_0, x_0]) \to [e, x_0]
$$
Hence,
$$
    \beta_{g_n g_{0}^{-1}}([g_0, x_0]) = [g_n, x_0] = [e, \al_{g_n}(x_0)] \to [e, x]
$$
which implies $\alpha_{g_n}(x_0) \to x$ and so $(X, \alpha, G)$ is transitive.
\end{proof}

\begin{remark} 
\emph{As a particular case of the proposition above we have the following:
$\overline{O^{\alpha}(x)} = X$ if and only if $\overline{O^{\beta}(x)} :=
\overline{\{ \beta_g(x) : g \in G \}} = X^e$.}
\end{remark}

The following definitions appear in \cite{E3}

\begin{definition}
\emph{(}i\emph{)} We say that a set $Y \subseteq X$ is $\al$-\emph{invariant} if
$\al_g(Y \cap X_{g^{-1}}) = Y \cap X_g$, for all  $g \in G$.

\emph{(}ii\emph{)} Let $(X,\alpha,G)$ be a partial dynamical system. An ideal $I$
of $C(X)$ is said to be $\alpha$-\emph{invariant} if $\alpha_g(I\cap C(X_{g^{-1}}))=
I\cap C(X_g)$, for all $g\in G$. Moreover, $C(X)$ is $\alpha$-simple
if the unique $\alpha$-invariant ideals are $(0)$ and $C(X)$.
\end{definition}

\begin{proposition}\label{dense2}
Let $(X, \al, G)$ be a minimal partial dynamical system.
Then, the unique $\al$-invariant open subsets of $X$ are $\emptyset$ and $X$.
\end{proposition}
\begin{proof}
Suppose that $X$ contains a proper $\alpha$-invariant open subset $U$.
Then there exists some $x_0\in X\backslash U$. Since $\overline{O^{\al}(x_0)} = X$
and $U$ is open, $U$ contains some point of $O^{\al}(x_0)$. Since $U$ is an
$\alpha$-invariant set, then it must contain all the orbit, a contradiction because
we are considering $x_0\in X\backslash U$.
\end{proof}

\begin{remark}\label{r1}
\emph{It is easy to see that for any  $\al$-invariant  ideal $I$ of $C(X)$ there exists
an open $\alpha$-invariant subset $Y \subseteq X$ such that $I = C(Y)$.}
\end{remark}

\begin{lemma}\label{7} The following conditions are equivalent:
\begin{itemize}
\item[$(i)$] $C(X)$ is $\alpha$-simple.
\item[$(ii)$] $X$ does not have proper $\alpha$-invariant closed subsets.
\item[$(iii)$] $(X,\alpha, G)$ is minimal.
\end{itemize}
\end{lemma}

\begin{proof}
$(i)\Rightarrow (ii)$

Suppose that $X$ contains a proper $\alpha$-invariant closed subset $S$.
We easily see that $X\backslash S$ is a proper
$\alpha$-invariant open subset of $X$. Hence,
$C(X\backslash S)$ is a proper $\alpha$-invariant ideal of $C(X)$,
which is a contradiction.

$ (ii)\Rightarrow (iii)$

Let $x$ be arbitrary element of $X$. Then we clearly have that $O^{\al}(x)$ is an
$\alpha$-invariant subset of $X$. It is not difficult to show that
$\overline{O^{\al}(x)}$ is a closed $\alpha$-invariant subset of $X$.
By assumption, we have that  $\overline{O^{\al}(x)}=X$. So $(X,\alpha,G)$ is minimal.

$(iii)\Rightarrow (i)$

Let $I$ be an $\al$-invariant ideal of $C(X)$. By Remark \ref{r1}, $I=C(U)$ for
some $\al$-invariant open subset $U \subseteq X$. By Proposition \ref{dense2},
we have that either $U=\emptyset$ or $U= X$ and so $I = (0)$ or $I = C(X)$.
\end{proof}

\subsection{Simplicity of $C(X)*_{\alpha}G$}

Throughout this subsection $(X, \alpha, G)$ is a partial dynamical system,
$C(X)$ is the algebra of continuous functions of $X$ over $\mathbb{C}$,
$\alpha$ will denote the extended partial action of $G$ on $X$ to $C(X)$,
and $C(X)*_{\alpha}G$ will be  the  partial skew group ring.

\begin{definition}
We define the \emph{commutant} of $C(X)$ in $C(X)\ast_{\alpha}G$ as
$$\mathcal{A}=\{a\in C(X)\ast_{\alpha}G:af=fa, \,\, \forall f\in C(X)\}.$$
\end{definition}

\begin{definition}
For any $g\in G\backslash \{e\}$, we set:
\begin{itemize}
\item[$(i)$] $Per^g_{C(X)}(X)=\{x\in X_g:f(x)=f(\alpha_{g^{-1}}(x)),
 \text{ for all } f\in C(X)\}$;
\item[$(ii)$] $Sep^g_{C(X)}(X)=\{x\in X_g: f(x)\neq f(\alpha_{g^{-1}}(x)),
 \text{ for some } f\in C(X)\}$.
\end{itemize}
\end{definition}

A topological space $X$ is said to be \emph{Hausdorff} if for any distinct
points $a,b\in X$, there exist open subsets $A$ and $B$ contained in $X$
such that $A\cap B=\emptyset$ with $a\in A$ and $b\in B$. Let us define,
for each $f\in C(X)$, the set $Supp(f)=\{x\in X:f(x)\neq 0\}$.

\begin{definition}
A  partial dynamical system $(X, \alpha, G)$ is said to be {\em topologically
free} if for each $g \in G\backslash \{e\}$, the set
$\theta_g=\{ a \in X_{g^{-1}} : \al_g(a) = a\}$ has empty interior.
\end{definition}

The following remark has indepedent interest.

\begin{remark}
\emph{Let $Gr(\alpha)=\{(t,x,y)\in G\times X\times X:x\in X_{t^{-1}},\alpha_t(x)=y\}$ 
and suppose that $Gr(\alpha)$ is closed. Then by  (\cite{Ab1}, Proposition 1.2) 
the enveloping space $X^e$ is a Hausdorff space and we easily have that $X$ is 
Hausdorff.  We claim that  for each $g\in G\backslash \{e\}$, $\theta_g$ is a 
closed set. In fact, for each sequence $\{x_n\}_{n\in \mathbb{N}}\subset \theta_g$ 
such that $x_n\rightarrow x$ we have  by the continuity of $i:X\rightarrow X$ and 
$\beta_g$, $g\in G$, that $\beta_g(i(x_n))\rightarrow \beta_g(i(x))$. By the fact 
that $\beta_g(i(x_n))=i(\alpha_g(x_n))=i(x_n)$, we obtain that
$i(x_n)\rightarrow \beta_g(i(x))$ and $i(x_n)\rightarrow i(x)$. Thus,
$\beta_g(i(x))=i(x)$.  Since $i(\alpha_g(x))=[e, x]=[g, x]=\beta_g(i(x))$, then $\alpha_g(x)=x$.}
\end{remark}

We have the following result, where item (ii) generalizes (\cite{SSJ}, Theorem 3.3).

\begin{theorem}\label{3}
\begin{itemize}
\item[$(i)$]
 The commutant of $C(X)$ in $C(X)\ast_{\alpha}G$ is
 $$
  \mathcal{A}=\{\sum_{g\in G}a_g\delta_g\in C(X)\ast_{\alpha}G:aa_g =
  \alpha_g(\alpha_{g^{-1}}(a_g)a), \forall  a\in C(X)\}.
 $$ Moreover, $C(X)\subseteq \mathcal{A}$.
\item[$(ii)$]
 The subalgebra
 $\mathcal{A}$ of $C(X)\ast_{\alpha}G$ is
 $$
  \mathcal{A}=
  \{\sum_{g\in G}a_g\delta_g\in C(X)\ast_{\alpha}G:a_g|_{Sep^g_{C(X)}(X)}\equiv0\}.
 $$
\item[$(iii)$]
 If $X$ is a Hausdorff space, then $$\mathcal{A}=
 \{\sum_{g\in G}a_g\delta_g\in C(X)\ast_{\alpha}G:Supp(a_g)\subseteq \theta_g\}.$$
\item[$(iv)$]
 Suppose that $G$ is abelian. Then $\mathcal{A}$ is the maximal commutative subalgebra
 of $C(X)\ast_{\alpha}G$ that contains $C(X)$.
\end{itemize}
\end{theorem}

\begin{proof}
(\emph{i}) The proof follows from Lemma \ref{lem1.2}.

(\emph{ii}) Let $\sum_{g\in G}a_g\delta_g\in \mathcal{A}$. Then by item (i),
$$
 \alpha_g(\alpha_{g^{-1}}(a_g)f)=fa_g,
$$
for any $f\in C(X)$. Note that for each $b\in Sep^g_{C(X)}(X)$ there exists $h\in C(X)$
such that $h(b)\neq h(\alpha_{g^{-1}}(b))$. Thus,
\begin{center}
$\alpha_g(\alpha_{g^{-1}}(a_g)h)(b)=(ha_g)(b) \Leftrightarrow
(\alpha_{g^{-1}}(a_g)h)(\alpha_{g^{-1}}(b))=h(b)a_g(b)\Leftrightarrow
\alpha_{g^{-1}}(a_g)(\alpha_{g^{-1}}(b))h(\alpha_{g^{-1}}(b))=h(b)a_g(b)\Leftrightarrow
a_g(b)h(\alpha_{g^{-1}}(b))=a_g(b)h(b)\Leftrightarrow
a_g(b)(h(\alpha_{g^{-1}}(b)-h(b))=0$.
\end{center}
Hence, $\alpha_g(b)=0$. So, $a_g|_{Sep^g_{C(X)}(X)}\equiv 0$.

(\emph{iii})
Let $\sum_{g\in G}a_g\delta_g$ be an element of $\mathcal{A}$. Then by item (i),
for any $f\in C(X)$ and $x\in X_g$, we have that
\begin{center}
$\alpha_g(\alpha_{g^{-1}}(a_g)f)(x)=(fa_g)(x) \Leftrightarrow
a_g(x)(f(\alpha_{g^{-1}}(x))-f(x))=0.$
\end{center}
Thus, for each $x\in X_g$ such that $a_{g}(x)\neq 0$, we have that
$f(\alpha_{g^{-1}}(x))=f(x)$. Since $X$ is Hausdorff, we obtain
$\alpha_{g^{-1}}(x)=x$, which implies that $x\in \theta_{g^{-1}}$.
By the fact that $\theta_{g^{-1}} = \theta_g$, we have that $x\in \theta_{g}$
and so $Supp(a_g)\subseteq \theta_g$.

(\emph{iv})
This proof is similar to the proof of (\cite{SSJ}, Proposition 2.1) and
we put it here for reader's convenience. We start by observing that every
commutative subalgebra of $C(X)\ast_{\alpha}G$ is contained in $\mathcal{A}$.
So it remains to show that $\mathcal{A}$ is commutative, which is a consequence
of Corollary \ref{cor1.8}.
\end{proof}

The next definition appears in (\cite{SSJ}).

\begin{definition} Let $B$ be a topological space and $\emptyset\neq A\subseteq B$.
Then $A$ is said to be a \emph{domain of uniqueness} for $B$ if for any continuous
function $f:B\rightarrow \mathbb{C}$  we have that $f|_{A}=0 \Rightarrow f=0$.
\end{definition}

Now we are in position to state the next result, where item (\emph{ii}) generalizes
(\cite{JO},  Lemma 8.2).

\begin{theorem} \label{5}
\begin{itemize}
\item[$(i)$]
 $C(X)=\mathcal{A}$  if and only if for any $g\in G\backslash \{e\}$,
 $Sep^g_{C(X)}(X)$ is a domain of uniqueness for $X_g$.
\item[$(ii)$]
 Suppose that $X$ is Hausdorff.    Then  $C(X)=\mathcal{A}$ if and only if  $(X,\alpha,G)$ is topologically free.
\end{itemize}
\end{theorem}
\begin{proof}
(\emph{i})
Suppose that $C(X)=\mathcal{A}$. Let $a_g\in X_g$, with $g\neq e$, such that
$a_g|_{Sep^g_{C(X)}(X)}=0$. Then, by Theorem \ref {3} item (ii),
$a_g\delta_g\in \mathcal{A}$. Hence, by assumption, $a_g=0$ and so,
$Sep^g_{C(X)}(X)$ is a domain of uniqueness for $X_g$, with $g\neq e$.

Conversely, suppose that for each $g\in G\backslash \{e\}$, $Sep^g|_{C(X)}(X)$ is
a domain of uniqueness for $X_g$.  Let  $\sum_{h\in G}a_h\delta_h\in \mathcal{A}$.
Then, by Theorem \ref{3}, we have that $a_h|_{Sep^h_{C(X)}(X)}=0$,
for each $h\in G\backslash \{e\}$. Hence, $a_h=0$, for all $h\in G\backslash \{e\}$.
So, $\mathcal{A}\subseteq C(X)$  and we have that $C(X)=\mathcal{A}$.

(\emph{ii}) Suppose that $(X,\alpha, G)$ is not topologically free, i.e.
that for some $g\in G\setminus\{e\}$ there exists a non-empty open subset $V$
contained in $\theta_g$. Thus $C(V)$ is a nonzero subalgebra of $C(X)$. Hence,
there exists $0\neq f\in C(V)$. Note that  
$Supp(f)\subseteq \theta_g$. So, by Theorem \ref{3}, item (iii),
$f\delta_g\in \mathcal{A}$, which contradicts the assumption.

Conversely, suppose that $C(X)\neq \mathcal{A}$. Then, there exists
$\sum_{j=1}^{n}a_{g_j}\delta_{g_j}\in \mathcal{A}$  with $g_i\neq e$
and $a_{g_i}\neq 0$, for some $1\leq i\leq n$. Let $x\in X_{g_i^{-1}}$
such that $a_{g_i}(x)\neq 0$. We claim that $Supp(a_{g_i})$ contains
an open set. In fact, suppose that for any open set $V\subseteq X_{g_i}$
we have that $V\nsubseteq Supp(a_{g_i})$. Hence, there exists
$x_1\in X_{g_i}$ such that $a_{g_i}(x_1)=0$. Since $X$ is Hausdorff,
there exists open subsets $A_1$ and $A_2$ of $X_{g_i}$ such that
$A_1\cap A_2=\emptyset$ with $x\in A_1$ and $x_1\in A_2$. By the
assumption on $a_{g_i}$ there exists $x_2\in A_1$ such that $a_{g_i}(x_2)=0$.
Proceeding by this way we can find a sequence $\{x_n\}_{n\in \mathbb{N}}\subset X_{g_i}$
such that $x_n\rightarrow x$ and $a_{g_i}(x_n)=0$. Since $a_{g_i}$ is continuous,
then $a_{g_i}(x_n)\rightarrow a_{g_i}(x)$, which implies $a_{g_i}(x)=0$, this is
a contradiction. Thus, there exists an open set
$A\subseteq Supp(a_{g_i})\subseteq \theta_{g_i}$, which contradicts the fact that
$(X,\alpha,G)$ is topologically free. So, $C(X)=\mathcal{A}$.
\end{proof}

\begin{definition}
Let $(X,\alpha,G)$ be a partial dynamical system. We say that a point $z\in X$ is
\emph{periodic} if there exists $g\in G\backslash \{e\}$ such that $z\in X_{g^{-1}}$
and $\alpha_{g}(z)=z$.
\end{definition}


\begin{lemma} \label{8}
Suppose that $X$ is infinite and the cardinality of the partial orbits of
periodic points is finite. If $(X,\alpha,G)$ is minimal, then $(X,\alpha, G)$
is topologically free.
\end{lemma}
\begin{proof}
Suppose that there exists $g\in G\backslash \{e\}$ such that
$\theta_g\neq \emptyset$. Then any point $x\in \theta_g$ has
a finite partial orbit. Since $(X,\alpha,G)$ is minimal, we have $\overline{O^{\alpha}(x)}=X$,
which contradicts the fact that $X$ is infinite.
\end{proof}


Now we are ready to prove the main result of this section that generalizes
(\cite{JO}, Theorem 8.6).

\begin{theorem} \label{10}
Let $(X,\alpha,G)$ be  partial dynamical system such that $X$ is an infinite
Hausdorff space and  the cardinality of the partial  orbits of periodic points
of $X$  is finite.   The following conditions are equivalent:
\begin{itemize}
\item[$(i)$] $(X,\alpha,G)$ is a minimal dynamical system.
\item[$(ii)$] $C(X)$ is maximal commutative in $C(X)\ast_{\alpha}G$ and $C(X)$
 is $\alpha$-simple.
\item[$(iii)$] $C(X)\ast_{\alpha}G$ is simple.
\end{itemize}
\end{theorem}

\begin{proof} $(i)\Rightarrow (ii)$

By the Lemma \ref{8}, $(X, \alpha,G)$ is topologically free and by Theorem \ref{5},
item (iii),  we have that $C(X)$ is maximal commutative. Since $(X,\alpha,G)$ is
minimal then  by Lemma \ref{7} we get that $C(X)$ is  $\alpha$-simple.

$(ii)\Rightarrow (iii)$

The proof follows from Theorem \ref{teo1.13}.

$(iii)\Rightarrow (i)$

For each $x\in X$, we have that $\overline{O^{\alpha}(x)}$ is an $\alpha$-invariant closed
subset of $X$. Thus $X\backslash \overline{O^{\alpha}(x)}$ is an $\alpha$-invariant
open subset of $X$. Suppose that $\overline{O^{\alpha}(x)}\varsubsetneq X$. Then
$C( X\backslash \overline{O^{\alpha}(x)})$ is a proper $\alpha$-invariant ideal of
$C(X)$, which contradicts the Lemma \ref{lem1.12}. So, $(X,\alpha,G)$ is minimal.
\end{proof}

\begin{remark}
\emph{It is convenient to point out if $G=\mathbb{Z}^n$, $n\geq 1$, in Lemma 3.16 we obtain that the
cardinality of the partial orbits is finite. Thus, in this case, we do not  need to assume the assumption
that cardinality of the partial orbits of periodic points of $X$ is finite in Theorem 3.17.}
\end{remark}

Next, we give an example to show that the assumption of the finiteness of the cardinality of the partial periodic points  in Theorem 3.17 is not superfluous.

\begin{example}

Let $X=(\mathbb{Z}_6)^{\mathbb{N}}$ be the topological space with the  discrete topology and the additive topological group $H=(\mathbb{Z}_6)^{\mathbb{N}} \times (\mathbb{Z}_6)^{\mathbb{N}}$ with product topology.  

We define the global (hence, partial ) action $\beta: H\times X\rightarrow X$ by \begin{center}$\beta_{((x_i)_{i\in \mathbb{N}}, (y_i)_{i\in \mathbb{N}})}((z_i)_{i\in \mathbb{N}})=(x_i)_{i\in \mathbb{N}}+(y_i)_{i\in \mathbb{N}}+(z_i)_{i\in \mathbb{N}}=(x_i+y_i+z_i)_{i\in \mathbb{N}}$.\end{center} We clearly have that this action is well defined.  Note that the element $(w_i)_{i\in \mathbb{N}}$ with $w_i=\overline{3}$, for all $i\in \mathbb{N}$, is periodic, because for $(x_i)_{i\in \mathbb{N}}$ and $(y_i)_{i\in \mathbb{N}}$ such that 
$x_i=y_i=\overline{3}$, $\forall i\in \mathbb{N}$ we have that  \begin{center}$\beta_{((x_i)_{i\in \mathbb{N}}, (y_i)_{i\in \mathbb{N}})}((w_i)_{i\in \mathbb{N}})=(x_i)_{i\in \mathbb{N}}+(y_i)_{i\in \mathbb{N}}+(w_i)_{i\in \mathbb{N}}=(x_i+y_i+w_i)_{i\in \mathbb{N}}=(w_i)_{i\in \mathbb{N}}$.\end{center} It is not difficult to see that the cardinality of the orbity $(w_i)_{i\in \mathbb{N}}$ is infinite, $X$ is Hausdorff and the unique $\beta$-invariant open sets are the trivial open sets, that is, $\emptyset$ and $X$. Hence, $X$ is minimal, but, $X$ is not topologically free because the set $\theta_g$, $g\in G$, has non-empty interior since the topology is discrete. So, by  Theorem 3.14, $C(X)\subsetneqq \mathcal{A}$. Therefore, the equivalent conditions on Theorem 3.17 does not hold in this case. 

\end{example} 
\section{Examples}

In this section, we present some examples which we apply some the principal results of
this article. All the examples of this section are build on metric spaces, and so they
are also Hausdorff spaces.

 \begin{example}\label{example4.2}
{\bf the horseshoe:} \emph{The horseshoe is a well known model in dynamical
systems theory; it appears naturally in systems presenting homoclinic points
and is the paradigm of the hyperbolic dynamical systems, see, for example,
\cite{S}. The dynamics is a diffeomorphism $F$ defined on the sphere $S^2$.
Typically one is interested on the restriction of this dynamics to the subset
$\cQ \subset S^2$ that is homeomorphic to the unitary square $Q = [0, 1]\times [0, 1]$.
Since this set is closed, we just relax the condition of $X_t$ being open sets on the
definition of a partial dynamical system to include also closed sets in what follows.}

\emph{In order to keep the presentation clear, we only describe the dynamics induced
by $F$ over the closed square $Q$ and we call it $f$,  assuming that this last one
is affine at some part of its domain. The diffeomorphism $f$ maps bijectively the
horizontal strips $[0, 1] \times [0, 1/3]$ and $[0, 1] \times [2/3, 1] $, respectively,
to the vertical strips  $[0, 1/3] \times [0, 1]$ and $[2/3, 1] \times [0, 1]$, the
horizontal strips being the domain where $f$ is affine.}

\emph{We can now see the horseshoe as a partial action of $\ZZ$ defined as follows: take
$$
X_n := Q \cap f^n(Q)   \; \text{and} \; \al_{n}(x):=f^n(x) \;
 \text{for}  \; n \in \ZZ, x \in X_{-n}
$$
Then $((X_n)_{n \in \ZZ}, (\al_n)_{n \in \ZZ}, \ZZ  )$ is a partial dynamical
system on the square $Q = X_0$.}

\emph{Since $\al_n$ is always affine on its domain, it is not hard to see that each
$\al_n$ has at most a finite number of fixed points for  any $n \neq 0$; hence, for
each $n \neq 0$ the set of the fixed points of $\al$ has empty interior.}

\emph{It is also possible to define a limit set $\La  = \bigcap_{n \in \ZZ} f^n(Q)$
that is homeomorphic to $\Si_2 = \{0, 1\}^{\ZZ}$. Over $\Si_2$ we can define an
homeomorphism known as shift, defined as follows:
$$
x \in \Si_2, x = (x_i)_{i \in \ZZ}, \; \text{then} \;\; (\si(x))_i = x_{i+1}
$$
This map is conjugate to $f$ restricted to $\La$, i.e. there exists an homeomorphism
$h \colon \La \to \Si_2$ such that $ h f = \si  h$. Hence, the restriction
$(\La, \al|_{\La}, \ZZ)$ is in fact a global action. By means of the conjugation
we get that  fixed points of $\al|_{\La}$ correspond to the fixed points of the shift
over $\Si_2$,  showing that they are finite for any $\al_n$, $n \in \ZZ$.}

\emph{Note that the dynamics of the horseshoe is topologically free (since the sets
of fixed points have empty interior) but it is not transitive: just take the open
balls  $B_r((1/2, 0.2))$ and $B_r((1/2, 0.8))$, for some positive $r < 1/10$; calling
one by $U$ and the other by $V$ its is easy to see that they violate the criterium
established in Theorem \ref{criteriotransitivo}.}

\emph{Since $(X_0, \alpha, \Z)$ it is not transitive, then it is not minimal and, by Lemma \ref{7},
$C(X_0)$ is not $\alpha$-simple. Hence, by Theorem \ref{10} we have that
$C(X_0)*_{\alpha}\mathbb{Z}$ is not simple.}
\end{example}

\begin{example}\label{example4.3}
\emph{We can use two dynamics $f$ and $g$ defined on the closed interval $[0, 1]$
and such that $f \circ g = g \circ f$, defined as follows:  $f$ has an interval of
fixed points, fix $0$ and $1$, $0$ is an attractor on the first interval and $1$
is an attractor on the third interval  (see the picture); for $g$  just take the
identity. Now we can consider a global (hence, partial) action  of $G = \ZZ^2$ on
$[0, 1]$ where $\al_{(m, n)}(x) = f^m \circ g^n(x)$. For $t = (0, 1)$, the interior 
of the $\theta_t$ is not empty and so the system is not topologically free. And the 
dynamics in fact is the dynamics of $f$, that is not transitive, since any open 
subset of the middle interval can not contain  points belonging to a dense orbit. 
In fact, it does not satisfy the criterium for transitivity of Theorem 
\ref{criteriotransitivo}, since given two open and disjoint subsets of the medium 
interval, $U$ and $V$, there exists no $g \in G$ such that
$\al_g (U) \cap V \neq \emptyset$}

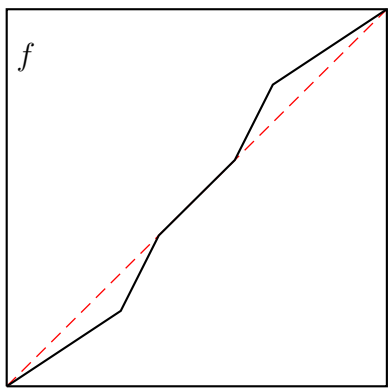
\begin{minipage}{12.5cm}
\begin{center}
  \psset{unit=0.5cm}
  \begin{pspicture}(-0.5,-0.5)(10,11)


    \psline[linestyle=dashed,linewidth=0.5pt,linecolor=red](0,0)(10,10)

    \psline(0,0)(10, 0)(10,10)(0,10)(0,0)

    \psline(0,0)(3,2)(4,4)(6,6)(7,8)(10,10)

     \uput[u] (0.5, 8){$f$}
    \end{pspicture}
 \psset{unit=1.0cm}
\end{center}
\end{minipage}
\vskip5mm%

\emph{So, by Theorem \ref{5}  and Theorem \ref{10} the algebra
$C(X)*_{\alpha}\mathbb{Z}^2$ is not simple.}
\end{example}

\begin{example}\label{example4.4}
\emph{Consider the map $R_{\om} \colon [0, 1] \to [0, 1]$  defined by
$R_{\om}(x)= (x+ \om) mod(1)$. It is well known that $R_{\om}$ has dense
orbits if and only if $\om$ is irrational. We use now $f=R_{\om}$, $\om$
irrational, $g$ the identity, and the set is $[0, 1]$ with $0$ identified
with $1$. Now we can consider, as in Example \ref{example4.3} above, the
global action of $G= \ZZ^2$ defined in the same way. Restricting this
global action to the set $(1/3, 2/3)$ we get a partial action, that is
not topologically free, has dense orbits, being transitive. So, by
Theorem \ref{5} and Theorem  \ref{10} we have that the algebra $C(X)*_{\alpha}\mathbb{Z}^2$ is
not simple.}
\end{example}

We finish this section with the following example that gives  an easy application of Theorems 3.14 and 3.17 to show the simplicity.

\begin{example} Let $G=\mathbb{R}^{*}$ be the multiplicative group  and $X=\mathbb{R}$  with usual topology. We consider the global (hence, partial) action $\alpha:G\times X\rightarrow X$ by $\alpha(x,y)=xy$. It is easy to see that the unique periodic element is $z=0$ and we clearly have that $X$ is topologically free. Moreover, the unique $\alpha$-invariant open subsets are the trivial  ones. So, by Theorems 3.14 and 3.17 we have that $C(\mathbb{R})*_{\alpha}\mathbb{R}^{*}$ is simple.
\end{example}

\end{document}